\pdfoutput=1
\RequirePackage{ifpdf}
\ifpdf % We~are running pdfTeX in pdf mode
\documentclass[pdftex]{sigma}
\else
\documentclass{sigma}
\fi

\DeclareMathOperator{\SG}{S}
\usepackage{mathtools}
\usepackage{xspace,bbm}
\usepackage{enumerate}

\usepackage{tikz}
\usepackage{gensymb}

\usepackage{bm,scalerel}
\usepackage[config, labelfont={bf},labelsep=period]{caption,subfig}
\usepackage{mathrsfs}
\usepackage{scalerel,stackengine}
\usetikzlibrary{matrix,arrows,calc}
\usetikzlibrary{decorations.markings}
\usetikzlibrary{patterns}
\usetikzlibrary{snakes}

\numberwithin{equation}{section}

\newtheorem{Theorem}{Theorem}[section]
\newtheorem*{Theorem*}{Theorem}
\newtheorem{Corollary}[Theorem]{Corollary}
\newtheorem{Lemma}[Theorem]{Lemma}
\newtheorem{Proposition}[Theorem]{Proposition}
 { \theoremstyle{definition}
\newtheorem{Definition}[Theorem]{Definition}

\newtheorem{Example}[Theorem]{Example}
\newtheorem{Remark}[Theorem]{Remark} }

\def\A{\text{\rm vN}(G_{\s,\q}(\HH_\R))}
\def\C{{\mathbb C}}
\def\R{{\mathbb R}}
\def\N{{\mathbb N}}

\def\D{{\rm D}}

\def\d{{\rm d}}
\def\6{\, {\rm d}}
\def\i{{\rm i}}

\def\ri{{\rm i}}

\def\B{b_{\s,\q}}
\def\G{G_{\s,\q}}

\def\r{b}
\def\ell {n}
\def\F{\mathcal{F}_{\rm fin}(\HH)}
\def\id{I}
\def\state {\varphi}
\def\m{\mu_{\s,\q,x,y}}

\def\qMP{{\rm MP}}
\def\P{\mathcal{P}}
\def\NC{\mathcal{NC}}

\def\Out{\rm Out}
\def\BL{{\mathbf{B}}}

\def\PB{\mathcal{P}^{B}}
\def\PA{\mathcal{P}^{A}}

\def\Cr{\text{\normalfont Cr}}
\def\InS{\text{\normalfont Cs}}

\def\NB{\text{\normalfont Nb}}

\def\Pair{\text{\normalfont Pair}}
\def\Sing{\text{\normalfont Sing}}
\renewcommand{\epsilon}{\varepsilon}

\def\H{{\bf \mathcal K}_{ n}}
\newcommand{\e}{{\epsilon}}
\def\HH{{\bf \mathcal K}}

\newcommand{\q}{q}
\newcommand{\s}{\alpha}%{q_{_{\text{-}}}}
\newcommand{\x}{\mathbf{x}_{\overline n} \otimes \mathbf{x}_n}
\DeclareMathOperator{\Part}{\mathcal{P}}

\newcommand\MatchingMeandersab[3]{%
	\begin{tikzpicture}[scale=0.4]
		\foreach \x in {1,...,#1}{
			\draw[circle,fill] (\x,0)circle[radius=1mm]node[below]{};
		}
		\foreach \x/\y in {#2} {
			\pgfmathsetmacro{\Radius}{\y/2-\x/2}
			\draw(\x,0) arc[radius=\Radius, start angle=180, end angle=0];
			;
		}
		\foreach \x/\y in {#3} {
			\pgfmathsetmacro{\Radius}{\y/2-\x/2}
			\draw(\x,0) arc[radius=\Radius, start angle=-180, end angle=0];
			;}
		\foreach \x in {-#1,...,-1}{
			\draw[circle,fill] (\x,0)circle[radius=1mm]node[below]{};
		}
		\node at (-4,-0.7) { $\overline 4$};
		\node at (-3,-0.7) { $ \overline 3$};
		\node at (-2,-0.7) { $\overline 2$};
		\node at (-1,-0.7) { $\overline 1$};
		\node at (4,-0.7) { $ 4$};
		\node at (3,-0.7) { $ 3$};
		\node at (2,-0.7) { $ 2$};
		\node at (1,-0.7) { $ 1$};
	\end{tikzpicture}
}

\newcommand\MatchingMeandersabc[3]{%
	\begin{tikzpicture}[scale=0.7]
		\foreach \x in {1,...,#1}{
			\draw[circle,fill] (\x,0)circle[radius=1mm]node[below]{};
		}
		\foreach \x/\y in {#2} {
			\pgfmathsetmacro{\Radius}{\y/2-\x/2}
			\draw(\x,0) arc[radius=\Radius, start angle=180, end angle=0];
			;
		}
		\foreach \x/\y in {#3} {
			\pgfmathsetmacro{\Radius}{\y/2-\x/2}
			\draw(\x,0) arc[radius=\Radius, start angle=-180, end angle=0];
			;}
		\foreach \x in {-#1,...,-1}{
			\draw[circle,fill] (\x,0)circle[radius=1mm]node[below]{};
		}
		\node at (-4,-0.5) { $\overline 4$};
		\node at (-3,-0.5) { $ \overline 3$};
		\node at (-2,-0.5) { $\overline 2$};
		\node at (-1,-0.5) { $ \overline 1$};
		\node at (4,-0.5) { $ 4$};
		\node at (3,-0.5) { $ 3$};
		\node at (2,-0.5) { $ 2$};
		\node at (1,-0.5) { $ 1$};
	\end{tikzpicture}
}

\newcommand\MatchingMeandersexample[3]{%
	\begin{tikzpicture}[scale=0.2]
		\foreach \x in {1,...,#1}{
			\draw[circle,fill] (\x,0)circle[radius=1mm]node[below]{};
		}
		\foreach \x/\y in {#2} {
			\pgfmathsetmacro{\Radius}{\y/2-\x/2}
			\draw(\x,0) arc[radius=\Radius, start angle=180, end angle=0];
			;
		}
		\foreach \x/\y in {#3} {
			\pgfmathsetmacro{\Radius}{\y/2-\x/2}
			\draw(\x,0) arc[radius=\Radius, start angle=-180, end angle=0];
			;}
		\foreach \x in {-#1,...,-1}{
			\draw[circle,fill] (\x,0)circle[radius=1mm]node[below]{};
		}
	\end{tikzpicture}
}

\newcommand\MatchingMeandersPositiveTracea[3]{%
	\begin{tikzpicture}[scale=0.35]
		\foreach \x in {1,...,#1}{
			\draw[circle,fill] (\x,0)circle[radius=1mm]node[below]{};
		}
		\foreach \x/\y in {#2} {
			\pgfmathsetmacro{\Radius}{\y/2-\x/2}
			\draw(\x,0) arc[radius=\Radius, start angle=180, end angle=0];
			;
		}
		\foreach \x/\y in {#3} {
			\pgfmathsetmacro{\Radius}{\y/2-\x/2}
			\draw(\x,0) arc[radius=\Radius, start angle=-180, end angle=0];
			;}
		\foreach \x in {-#1,...,-1}{
			\draw[circle,fill] (\x,0)circle[radius=1mm]node[below]{};
		}
		\node at (-6.7,-0.6) { $\scriptscriptstyle{\bar v_m}$};
		\node at (-3.8,-0.6) { $\scriptscriptstyle{ \overline v_{i}}$};
		\node at (-0.8,-0.6) { $\scriptscriptstyle{\bar v_{1}}$};
		
		\node at (7.4,-0.6) { $\scriptscriptstyle{ v_m}$};
		\node at (4.2,-0.6) { $ \scriptscriptstyle{ v_{i}}$};
		\node at (1.2,-0.6) { $\scriptscriptstyle{ v_1}$};
		\node at (-2.5,4.5) { $\scriptscriptstyle{ \overline{W}_j}$};
		\node at (2.5,4.5) { $\scriptscriptstyle{ {W}_j}$};
		\node at (9,4) { $ \mapsto$};
	\end{tikzpicture}
}

\newcommand\MatchingMeandersPositiveTraceb[3]{%
	\begin{tikzpicture}[scale=0.35]
		\foreach \x in {1,...,#1}{
			\draw[circle,fill] (\x,0)circle[radius=1mm]node[below]{};
		}
		\foreach \x/\y in {#2} {
			\pgfmathsetmacro{\Radius}{\y/2-\x/2}
			\draw(\x,0) arc[radius=\Radius, start angle=180, end angle=0];
			;
		}
		\foreach \x/\y in {#3} {
			\pgfmathsetmacro{\Radius}{\y/2-\x/2}
			\draw(\x,0) arc[radius=\Radius, start angle=-180, end angle=0];
			;}
		\foreach \x in {-#1,...,-1}{
			\draw[circle,fill] (\x,0)circle[radius=1mm]node[below]{};
		}
		\node at (-6.7,-0.6) { $\scriptscriptstyle{\bar v_m}$};
		\node at (-3.8,-0.6) { $\scriptscriptstyle{ \overline v_{i}}$};
		\node at (-0.8,-0.6) { $\scriptscriptstyle{\bar v_{1}}$};
		
		\node at (7.4,-0.6) { $\scriptscriptstyle{ v_m}$};
		\node at (4.2,-0.6) { $ \scriptscriptstyle{ v_{i}}$};
		\node at (1.2,-0.6) { $\scriptscriptstyle{ v_1}$};
		\node at (-10,-0.6) { $\scriptscriptstyle{\overline{k+1}}$};
		\node at (10,-0.6) { $\scriptscriptstyle{k+1}$};
		\node at (-2.5,4.5) { $\scriptscriptstyle{ \overline{W}_j}$};
		\node at (2.5,4.5) { $\scriptscriptstyle{ {W}_j}$};
	\end{tikzpicture}
	
}

\newcommand\MatchingMeanders[2]{%
	\begin{tikzpicture}[scale=0.7]
		\foreach \x in {1,...,#1}{
			\draw[circle,fill] (\x,0)circle[radius=1mm]node[below]{$ \x$};
		}
		\foreach \x/\y in {#2} {
			\pgfmathsetmacro{\Radius}{\y/2-\x/2}
			\draw(\x,0) arc[radius=\Radius, start angle=180, end angle=0];
			;
			\node at (0,5.45) { $\scriptstyle{\blacklozenge}$};
			\node at (0,4.2) { $\scriptstyle{\blacklozenge}$};
			\node at (0,1.7) { $\scriptstyle{\blacklozenge}$};
			\draw[] (0,0) -- (0,6);
		}
		
		\foreach \x in {#1,...,1}{
			\draw[circle,fill] (-\x,0)circle[radius=1mm]node[below]{$ \overline{\x}$};
		}
	\end{tikzpicture}
}

\begin{document}

\allowdisplaybreaks

\newcommand{\arXivNumber}{2202.13762}

\renewcommand{\thefootnote}{}

\renewcommand{\PaperNumber}{040}

\FirstPageHeading

\ShortArticleName{The Double Fock Space of Type B}

\ArticleName{The Double Fock Space of Type B\footnote{This paper is a~contribution to the Special Issue on Non-Commutative Algebra, Probability and Analysis in Action. The~full collection is available at \href{https://www.emis.de/journals/SIGMA/non-commutative-probability.html}{https://www.emis.de/journals/SIGMA/non-commutative-probability.html}}}

\Author{Marek BO\.ZEJKO~$^{\rm a}$ and Wiktor EJSMONT~$^{\rm b}$}

\AuthorNameForHeading{M.~Bo\.zejko and W.~Ejsmont}

\Address{$^{\rm a)}$~Institute of Mathematics, University of Wroc{\l}aw, Pl.\ Grunwaldzki 2/4,\\
\hphantom{$^{\rm a)}$}~50-384 Wroc{\l}aw, Poland}
\EmailD{\href{marek.bozejko@math.uni.wroc.pl}{marek.bozejko@math.uni.wroc.pl}}

\Address{$^{\rm b)}$~Department of Telecommunications and Teleinformatics, Wroc{\l}aw University of Science\\
\hphantom{$^{\rm a)}$}~and Technology, Wybrze\.ze Wyspia\'nskiego 27, 50-370 Wroc{\l}aw, Poland}
\EmailD{\href{wiktor.ejsmont@gmail.com}{wiktor.ejsmont@gmail.com}}
\URLaddressD{\url{https://sites.google.com/site/wiktorejsmont/r-i?authuser=0}}

\ArticleDates{Received March 28, 2022, in final form May 25, 2023; Published online June 12, 2023}

\Abstract{In this article, we introduce the notion of a double Fock space of type B. We will show that this new construction is compatible with combinatorics of counting positive and negative inversions on a hyperoctahedral group.}


\Classification{6L53; 47N30; 46L54}

 \begin{flushright}
 \begin{minipage}{78mm}
 \textit{We dedicate this paper to Michael Sch\"{u}rmann\\ on the occasion of his retirement}
 \end{minipage}
 \end{flushright}

\renewcommand{\thefootnote}{\arabic{footnote}}
\setcounter{footnote}{0}

\section{Introduction}

Several deformations of boson, fermion and full Fock spaces and Brownian motion
have been proposed so far. Bo\.zejko and Speicher used the Coxeter groups of type A ($={}$the symmetric group) to
construct a $q$-deformed Fock space and a $q$-deformed Brownian motion \cite{BozejkoSpeicher1991} ($={}$Fock space of type~A).
Bo\.zejko, Ejsmont and Hasebe followed this idea
in~\cite{BozejkoEjsmontHasebe2015} and constructed an
$(\alpha, \q)$-Fock space using the Coxeter groups of type B.
We give an alternative construction to \cite{BozejkoEjsmontHasebe2015} of a generalized Gaussian process related to the Coxeter groups of type~B.
Our motivation is inspired by the following reasoning.
If we start the construction of a deformed probability space by using some symmetrization
operator (with a given statistic on the set of permutations), then we obtain that the joint moments of a Gaussian operator may be expressed by the analogue of statistic on the set of pair partitions.
We explain this by the examples.

First of all, we focus on the work by Bo\.zejko and Speicher about $q$-Gaussian process \cite{BozejkoSpeicher1991} on the $q$-deformed Fock space
$\mathcal{F}_q(H):=(\mathbb{C}\Omega)\oplus\bigoplus_{n=1}^\infty H^{\otimes n} $, where $-1\leq q \leq 1$, and $\Omega$ denotes the vacuum vector and $H$ is the complexification of some real separable
Hilbert space $H_\R$. On this space the authors introduced
a deformed inner product, using the following symmetrization:
\begin{align*}
	\sum_{\sigma\in \SG(n)}q^{\operatorname{inv}(\sigma)}\sigma,
\end{align*}
where $\SG(n)$ is the set of all the permutations of $\{1,n,\dots,n\}$ (the Coxeter groups of type A) and $\operatorname{inv}(\sigma):=\operatorname{card}\{(i,j)\colon i<j,
\sigma(i)>\sigma(j)\}$ is the number of inversions of $\sigma\in
\SG(n)$.
On this space we define a creation $a_q^\ast(x)$ and its adjoint, i.e.,\ an annihilation operator $a_q(x)$, and a Gaussian operator $G_q(x):=a_q^*(x)+a_q(x)$.
The key point of $q$-Gaussian distributions is the methodology of computing their multidimensional moments, which are given by
\begin{equation*}%\label{mixed moment}
	\begin{split}
		&\big\langle\Omega, G_q(x_1)\cdots G_q(x_{2n})\Omega\big\rangle_{q}=
		\sum_{\pi\in \P_{2}(2 n)} q^{cr(\pi)} \prod_{\substack{\{i,j\} \in \pi} }\langle x_i,x_j\rangle,
	\end{split}
\end{equation*}
where $cr(\pi)$ is the number of crossings of a pair partition $\pi$ (see \cite{BozejkoSpeicher1991}).

Bo\.{z}ejko and Yoshida \cite{BozejkoYoshida2006} (see also \cite{Blitvic2012}) introduced a two-parameter refinement of the $q$-Fock space,
formulated as a $(q,t)$-Fock space $\mathcal{F}_{q,t} (H)$, where $-1\leq q, t \leq 1$.
The corresponding symmetrization
operator has the form
\begin{align*}
\sum_{\sigma\in \SG_n}q^{\operatorname{inv}(\sigma)}t^{\operatorname{cinv}(\sigma)}\sigma ,
\end{align*}
where $\operatorname{cinv}(\sigma)$ denotes the number of co-inversions of a permutation $\operatorname{cinv}(\sigma):=\operatorname{card}\{(i,j)\colon {i<j},\allowbreak \sigma(i)<\sigma(j)\}$.
If we denote the creation operator by $a^\ast_{q,t}(x)$ and the annihilation operator by~$a_{q,t}(x)$, then
the moments of
the deformed Gaussian operator $G_{q,t}(x):=a_{q,t}^*(x)+a_{q,t}(x)$ are encoded by the joint statistics of crossings
and nestings in the set of pair partitions:
\begin{equation*}%\label{mixed moment}
	\begin{split}
		&\big\langle\Omega, G_{q,t}(x_1)\cdots G_{q,t}(x_{2n})\Omega\big\rangle_{q}=
		\sum_{\pi\in \P_{2}(2 n)} q^{cr(\pi)}t^{\operatorname{nest}(\pi)} \prod_{\substack{\{i,j\} \in \pi} }\langle x_i,x_j\rangle,
	\end{split}
\end{equation*}
where $\operatorname{nest}(\pi)$ denotes the number of nestings in pair partition of $\pi$ (see \cite[Section 2]{Blitvic2012}).

From these two examples we see that
the inversions and co-inversions in the symmetric group now become
the joint statistics of crossings and nestings on the set of pair partitions.
The symmetrization operator of type B for $ -1\leq \alpha, q \leq 1$, has the form
\begin{align*}
\sum_{\sigma \in B(n)}\s^{\text{number of negative inversions in }\sigma}
\q^{\text{\,number of positive inversions in }\sigma}\sigma,
\end{align*}
where \begin{itemize}\itemsep=0pt
	\item $B(n)$ is the hyperoctahedral group, see Sections~\ref{sec:preliminaries} below;
	\item the definition of positive and negative inversion is given in Section~\ref{positivenegativeinversions} below.
\end{itemize}
If we apply the above reasoning to the statistic which appears during calculation of moments of type-B Gaussian operator, it should be
\[
\sum_{\pi\in\PB_{2}(2n)} \alpha^{\text{number of negative pairs of }\pi}q^{\text{numbers of crossings of }\pi},
\]
where $\PB_{2}(2n)$ is the set of pair partitions of type B (see Definition~\ref{def:partycji}).
It is worth mentioning that probabilistic considerations of type B that first appeared in \cite[Corollary~3.9\,(2)]{BozejkoEjsmontHasebe2015} are not quite associated with our present approach.
The goal of this paper is to introduce a double Fock space of type B, such that the Gaussian operator moments are compatible with a statistic symmetrization related to the Coxeter groups of type B. We claim that this approach is more natural than the methodology proposed in \cite{BozejkoEjsmontHasebe2015}.
In particular, we use this model in order to describe the kernel of the symmetrization operator, that was difficult to achieve with the model from \cite{BozejkoEjsmontHasebe2015}.

\section{Preliminaries }
\label{sec:preliminaries}
In the present section we show that the construction of type B Fock space \cite{BozejkoEjsmontHasebe2015} can be adapted to recover the
double Fock space of type B.
In the following we will briefly describe the tools we use in this investigation,
which is partially contained in the previous paper \cite{BozejkoEjsmontHasebe2015}.
For further information the reader is referred to \cite{BozejkoEjsmontHasebe2015} and the references therein.

\subsection{Coxeter groups of type B}
Recall that the Coxeter
group of type B (also known as the hyperoctahedral group) of degree $n$, denoted by $B(n)$, is generated
by the elements $\pi_0, \pi_1, \dots , \pi_{n-1}$ subject to the defining relations $\pi_i^2=e$, $0\leq i \leq n-1$, $(\pi_0\pi_1)^4=(\pi_i \pi_{i+1})^3=e$, $1\leq i < n-1$ and $(\pi_i \pi_j)^2=e$ if $|i-j|\geq2$, $0\leq i,j\leq n-1$. Note that $\{\pi_i\mid i=1,\dots,n-1\}$ generate the symmetric group $S(n)$.
The Coxeter diagram for $B(n)$ is described in Figure \ref{fig:BN}.
\begin{figure}[htp]
\centering
		\begin{tikzpicture}
			[scale=.5,auto=left,every node/.style={circle}]
			\node (n7) at (1,3.6) {$\pi_{n-1}$};
			\node (n6) at (1,3) {$\circ$};
			\node (n5) at (-3,3) {$\dots$};
			\node (n3) at (-7,3.6) {$\pi_2$};
			\node (n1) at (-7,3) {$\circ$};
			\node (n4) at (-11,3.6) {$\pi_{1}$};
			\node (n2) at (-11,3) {$\circ$};
			\node (n8) at (-15,3.6) {$\pi_{0}$};
			\node (n0) at (-15,3) {$\circ$};
			
			\foreach \from/\to in {n2/n1, n1/n2,n1/n5,n5/n6}
			\draw (\from) -- (\to);
			\draw (-14.5,3.1) -- (-11.5,3.1);
			\draw (-14.5,2.9) -- (-11.5,2.9);
		\end{tikzpicture}
		\caption{The Coxeter diagram for $B(n)$.}
		\label{fig:BN}
\end{figure}
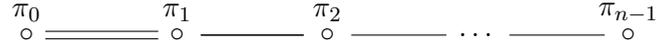

We express $\sigma \in B(n)$ in an irreducible form
\[
\sigma=\pi_{i_1} \cdots \pi_{i_{k}}, \qquad 0\leq i_1,\dots,i_k \leq n-1,
\]
i.e., in a form with the minimal length, and in this case let
\begin{align*}
	&l_1(\sigma)= \text{the number of the occurrences of the factor $\pi_0$ in $\sigma$},
\\[1mm]
	&l_2(\sigma) = \text{the number of the occurrences of all the factor of the form $\pi_1,\dots,\pi_{n-1}$ in $\sigma$}.
\end{align*}

\begin{Remark}
These definitions do not depend on the way we express $\sigma$ in an irreducible form, and therefore, $l_1(\sigma)$ and $l_2(\sigma)$ are well defined (see \cite[Proposition~1]{BozejkoSzwarc2003}).
\end{Remark}

\subsection[(q,alpha)-Meixner--Pollaczek orthogonal polynomials]{$\boldsymbol{(\q,\s)}$-Meixner--Pollaczek orthogonal polynomials}
In this subsection, we remind basic facts about the
orthogonal polynomials.

For a probability measure $\mu$ with finite moments of all orders, we can assign the orthogonal polynomials $(P_n(x))_{n=0}^\infty$
with $\operatorname{deg} P_n(x) =n$
and the leading coefficient of each $P_n(x)$ is $1$, i.e., monic.
It is known that they satisfy the recurrence relation
\begin{align*} %\label{wielortogonalnerekursia}
	x P_n(x) = P_{n+1}(x) +\beta_n P_n(x) + \gamma_{n-1} P_{n-1}(x),\qquad n =0,1,2,\dots,
\end{align*}
with the convention that $P_{-1}(x)=0$. The coefficients $\beta_n$ and $\gamma_n$ are called \emph{Jacobi parameters} and they satisfy $\beta_n \in \R$ and $\gamma_n \geq 0$.
It is known that
\begin{equation*}%\label{eq54}
	\gamma_0 \cdots \gamma_n=\int_{\R}|P_{n+1}(x)|^2\mu({\rm d} x),\qquad n \geq 0.
\end{equation*}
Let
\begin{itemize}\itemsep=0pt
	\item $[n]_q$ be the $q$-number	$[n]_\q:= 1+\q+\cdots+\q^{n-1}$, $n \geq1$;

	\item $[n]_\q!$ be the $\q$-factorial $	[n]_\q !:= [1]_q \cdots [n]_\q$, $n \geq1$;

	\item $(s;\q)_n$ be the $\q$-Pochhammer symbol 	$(s;q)_n:= \prod_{k=1}^n\big(1-s \q^{k-1}\big)$,
$s \in \R$, $|\q|<1$, $n \geq1$.
\end{itemize}

Let $\qMP_{\alpha,q}$ be the probability measure supported on $\bigl(-2/\sqrt{1-q}, 2/\sqrt{1-q}\bigr)$.
This measure has density (with respect to the Lebesgue measure)
\begin{equation}\label{eq10}
\frac{{\rm d} \qMP_{\alpha,q}}{{\rm d}t}(t)= \frac{(q;q)_\infty\big(\beta^2; q\big)_\infty}{2\pi\sqrt{4/(1-q) -t^2}}\cdot\frac{g(t,1;q) g(t,-1;q) g(t,\sqrt{q};q) g(t,-\sqrt{q};q)}{g(t, \i \beta;q)g(t,-\i \beta;q)},
\end{equation}
where
\begin{gather*}
g(t,b;q)= \prod_{k=0}^\infty\big(1-4 b t (1-q)^{-1/2} q^k + b^2 q^{2k}\big), \qquad
	(s;q)_\infty=\prod_{k=0}^\infty\big(1-s q^{k}\big),\qquad s \in \R,
\\
\beta
	=
	\begin{cases}
		\sqrt{-\alpha}, & \alpha \leq 0, \\
		\ri \sqrt{\alpha}, & \alpha \geq0.
	\end{cases}
\end{gather*}

\begin{Remark}
In equation~\eqref{eq10}, we assume that $({\alpha,q})\in (-1,1)\times (-1,1)$, but
	by weak continuity we may allow the parameters $({\alpha,q})$ of $\qMP_{\alpha,q}$ to take any values	in $[-1,1]\times [-1,1]$.
\end{Remark}

\begin{Example}%\label{przyklad1}
In the special case, we have
\begin{enumerate}\itemsep=0pt
		\item The measure $\qMP_{\alpha,1}$ is the normal law $(2(1+\alpha)\pi)^{-1/2}{\rm e}^{-\frac{t^2}{2(1+\alpha)}}1_\R(t)\,\d t$.
		
		\item The measure $\qMP_{0,0}$ is the standard Wigner's semicircle law $(1/2\pi)\sqrt{4-t^2}1_{(-2,2)}(t)\,\d t$.
		
		\item The measure $\qMP_{0,q}$ is the $q$-Gaussian law~\cite{BozejkoSpeicher1991}.
		
		\item The measure $\qMP_{\alpha,-1}$ is the Bernoulli law
		$(1/2)\big(\delta_{\sqrt{1+\alpha}}+\delta_{-\sqrt{1+\alpha}}\big)$.
		
		\item The measure $\qMP_{\alpha,0}$ is a symmetric free Meixner law~\cite{A03,BozejkoBryc2006,Ejsmont2013}.
	\end{enumerate}
\end{Example}

The orthogonal polynomials $(Q_n^{{\s,\q}}(t))_{n=0}^\infty$ associated to
the distribution of the
$\qMP_{\alpha,q}$ are called
\emph{$(\q,\s)$-Meixner--Pollaczek polynomials}
satisfying the recurrence relation
\begin{equation}\label{recursion}
	t Q_n^{(\s, \q)}(t) = Q_{n+1}^{(\s, \q)}(t) +[n]_q\big(1 + \s \q^{n-1}\big)Q_{n-1}^{(\s, \q)}(t), \qquad n=0,1,2,\dots,
\end{equation}
where $Q_{-1}^{{\s,\q}}(t)=0$, $Q_0^{{\s,\q}}(t)=1$ and $-1 \leq \s,\q \leq 1$.

\section{The double Fock space of type B}
%\label{subsec:Fock}

Let $H_\R$ be a separable real Hilbert space and let $H$ be its
complexification with the inner product~$\langle\cdot,\cdot\rangle$, linear in the right component and anti-linear in the left.
The Hilbert space $\HH:=H\otimes {H}$ is the complexification of its
real subspace $\HH_\R:=H_\R\otimes {H}_\R$, with the inner product
\[
\langle x\otimes y,\xi\otimes \eta \rangle_{\HH} = \langle
x,\xi\rangle\langle y,\eta\rangle.
\]

We define $\H:=H^{\otimes n} \otimes {{H}^{\otimes n}}=H^{\otimes 2n}$ and instead of
indexing its simple tensors by $\{1,\dots,2n\}$ we will index them by $[\pm n]=\{-n,\dots,-1,1,\dots,n\}$ and we will use the
typical involution $\bar n=-n$ for $n\in \N$
\[
\H\ni\mathbf{x}_{\overline n} \otimes \mathbf{x}_n=x_{\overline n}\otimes\cdots \otimes x_{\overline 1} \otimes x_1\otimes\cdots \otimes x_n=x_{\overline n}\otimes \cdots \otimes x_{n}.
\]
We use this convention for indexing the elements of $\H$ to define a natural action of the
hyperoctahedral group $B(n)$ on $\H$ by setting
\begin{align*}
\sigma\colon\, &\H\to \H,
\\
	&x_{\overline n}\otimes\cdots \otimes x_{\overline 1} \otimes x_1\otimes\cdots \otimes x_n\mapsto x_{\sigma^{-1}(\overline n)}\otimes\cdots \otimes x_{\sigma^{-1}(\overline 1)} \otimes x_{\sigma^{-1}(1)}\otimes\cdots \otimes x_{\sigma^{-1}(n)}
\end{align*}
for any $\sigma \in B(n)$.

\begin{Remark}
	It is worth to mention that in the present paper we act on the indexes, but in the previous one, \cite{BozejkoEjsmontHasebe2015}, we were acting on the vectors.
\end{Remark}

Let $\F$ be the algebraic full Fock space over $\HH$:
\begin{equation*}
	\F:= \bigoplus_{n=0}^\infty \H= \bigoplus_{n=0}^\infty H^{\otimes 2n},
\end{equation*}
with the convention that $\HH^{\otimes 0} =H^{\otimes 0} \otimes H^{\otimes 0}=\C\Omega \otimes \Omega$ is the one-dimensional normed space along with the unit vector $\Omega \otimes \Omega$.
We equip $\F$ with the inner product
\[
\langle x_{\overline n} \otimes \cdots \otimes x_{\overline 1} \otimes x_1 \otimes \cdots \otimes x_n, y_{\overline m} \otimes \cdots \otimes y_{\overline 1} \otimes y_1 \otimes \cdots \otimes y_m\rangle_{0,0}:= \delta_{m,n}\prod_{i=\overline n }^n \langle x_i, y_i\rangle.
\]
By definition, the action of the generators $\pi_i$ for $0\leq i \leq n-1$ on
\begin{align*}
\eta = x_{\bar n} \otimes \cdots \otimes x_{\bar 1} \otimes x_1 \otimes \cdots \otimes x_{n}\in \H
\end{align*}
is given for $i\geq 1$ and $n\geq 2$ by
\begin{align*}
\pi_i(\eta) =
	x_{\bar n} \otimes \cdots \otimes x_{\overline{i}} \otimes x_{\overline{i+1}} \otimes \cdots \otimes x_{\bar 1} \otimes x_1 \otimes \cdots \otimes x_{i+1} \otimes x_{i} \otimes \cdots \otimes x_{n}, \end{align*}
and for $i=0$ and $n\geq 1$ by
\begin{align*}
\pi_0(\eta) = x_{\bar n} \otimes \cdots \otimes x_1 \otimes x_{\bar 1}\otimes \cdots \otimes x_{n}.
\end{align*}

For the parameters $ -1\leq \alpha, q \leq 1$, we define the symmetrization operators
\begin{align*}
	&P_{\s,\q}^{(n)}:= \sum_{\sigma \in B(n)}\s^{l_1(\sigma)} \q^{l_2(\sigma)} \, \sigma,\qquad n \geq1, \\
	&P_{\s,\q}^{(0)}:= \id_{H^{\otimes 0}\otimes H^{\otimes 0}}.
	\intertext{Moreover, let }
	&P_{\s,\q}:=\bigoplus_{n=0}^\infty P_{\s,\q}^{(n)}
\end{align*} be the \emph{type B symmetrization operator} acting on the algebraic full Fock space.
\begin{Remark}
	From
	Bo\.zejko and Speicher \cite[Theorem 2.1]{BozejkoSpeicher1994}, the operator $P_{\s,\q}^{(n)}$ is positive for $ -1\leq \alpha, q \leq 1$. If $ -1< \alpha, q <1$, then $P_{\s,\q}^{(n)}$ is a strictly positive operator meaning that it is positive and
	$\ker P_{\s,\q}^{(n)}=\{0\}$.
\end{Remark}
For $\x\in \HH_{ n}$ and $\mathbf{y}_{\overline m} \otimes \mathbf{y}_m\in
\HH_{ m}$, we deform the inner product $\langle\cdot,\cdot\rangle_{0,0}$ by using the type~B symmetrization operator:
\begin{align*}
	\langle \x,\mathbf{y}_{\overline m} \otimes \mathbf{y}_m\rangle_{\s,\q}:=\delta_{n,m}\langle \x,P_{\s,\q}^{(m)}\mathbf{y}_{\overline m} \otimes \mathbf{y}_m\rangle_{0,0}.
\end{align*}

For $x\in H$, let $l(x)$ and $r(x)$ be the free left and right
annihilation operators on $H^{\otimes n} $, respectively, defined by
the equations
\begin{gather*}
l^\ast(x)(x_1\otimes \dots \otimes x_n ):=x\otimes x_1\otimes \dots \otimes x_n ,
\\
l(x)(x_1\otimes \dots \otimes x_n ):=\langle x, x_1 \rangle x_2\otimes \dots \otimes x_n ,
\\
r^\ast(x)(x_1\otimes \dots \otimes x_n ):= x_1\otimes \dots \otimes x_n \otimes x,
\\
r(x)(x_1\otimes \dots \otimes x_n ):=\langle x, x_n \rangle x_1\otimes \dots \otimes x_{n-1},
\end{gather*}
where the adjoint is taken with respect to the free inner product.
The left-right creation and annihilation operators $\r^\ast(x\otimes y)$, $\r(x\otimes y)$ on $\F$ are defined as follows:
\begin{alignat*}{2}
&\r^\ast(x\otimes y)(\x):=l^\ast(x)\mathbf{x}_{\overline n}\otimes
	r^\ast( y) \mathbf{x}_n, \qquad &&\r^\ast(x\otimes y)\Omega \otimes \Omega:=x \otimes y,
\\
&\r(x\otimes y)(\x ):= l( x) \mathbf{x}_{\overline n}\otimes r(y)\mathbf{x}_n , \qquad &&\r(x\otimes y)\Omega \otimes \Omega :=0,
\end{alignat*}
where $\x\in \H$, $n\geq 1$.
It holds that $[\r^\ast(x\otimes y)]^\ast = \r(x\otimes y)$ where the adjoint is taken with respect to $\langle \cdot, \cdot \rangle_{0,0}$.

\begin{Definition} Let $\q,\s \in (-1,1)$. The algebraic full Fock space $\F$ equipped with the inner product $\langle\cdot,\cdot \rangle_{\s,\q}$ is called the \emph{double Fock space of type B} and denoted by $\mathcal{F}_{\s,\q}(\HH)$.
	For $x\otimes y \in \HH$ we define $\B^\ast(x\otimes y):=
	\r^\ast(x\otimes y)$ and we consider its adjoint operator $\B(x\otimes
	y)$ with
	respect to the inner product $\langle\cdot,\cdot \rangle_{\s,\q}$
	acting on the Hilbert space $\mathcal{F}_{\s,\q}(\HH)$. The operators $\B^\ast(x\otimes y)$ and $\B(x\otimes y)$ are called \emph{double creation and double annihilation operator of type B, respectively}.
\end{Definition}

\begin{Remark}
We would like to emphasize that our double Fock space of type B is different from~\cite{BozejkoEjsmontHasebe2015}.
	In the present version of the model, we apply a more natural operation on tensor product, which finally contributes to the more natural combinatorics of calculating the Gaussian moments.
	In~\cite{BozejkoEjsmontHasebe2015}, we put the action of the symmetrization on $n$ points but in the current ap\-proach we give $2n$ points which is close to our construction from the articles~\cite{BDEG2021,Ejsmont2020}. In~\cite{BozejkoEjsmontHasebe2015}, we defined the creator in the same way as in the case of type A, and then we did not obtain a~natural combinatorics which is compatible with partitions of the set $\{\pm 1,\dots, \pm n\}$.
	Now if we start our construction with the creator operator of double type B, the situation is completely different from~\cite{BozejkoEjsmontHasebe2015}.
\end{Remark}
The following proposition can be derived directly
from \cite[Proposition 2.3]{BozejkoEjsmontHasebe2015}.

\begin{Proposition}%\label{prop1}
We have the decomposition
\begin{equation*}%\label{decomposition}
P^{(n)}_{\s,\q}=\big( \id \otimes P^{(n-1)}_{\s,\q}\otimes \id \big)R^{(n)}_{\s,\q}\qquad \text{on}\quad \H, \qquad n\geq 1,
\end{equation*}
where\vspace{-2mm}
\begin{gather*}
R^{(n)}_{\s,\q} = \id+\sum_{k=1}^{n-1}\q^{k}\pi_{n-1}\cdots \pi_{n-k} + \s \q^{n-1}\pi_{n-1} \pi_{n-2} \cdots \pi_{1}\pi_0\bigg(1+\sum_{k=1}^{n-1}\q^{k}\pi_{1}\cdots \pi_{k}\bigg), \quad n\geq 2,
\\
R^{(1)}_{\s,\q} =\id+\s\pi_0.% \quad n=1.
\end{gather*}
\end{Proposition}

We can compute the annihilation operator in terms of $R_{\s,\q}^{(n)}$ by using the similar
method to that in \cite[Proposition~2.4]{BozejkoEjsmontHasebe2015}.

\begin{Proposition}%\label{prop2}
For $n \geq1$ and $x\otimes y\in \HH_\R$, we have\vspace{-1mm}
\begin{equation*}
\B(x\otimes y)=\r(x\otimes y) R^{(n)}_{\s,\q}\qquad \text{on}\quad \H.
\end{equation*}
\end{Proposition}

\begin{Remark}\label{thm1}
	Using the above notation, we can decompose $\B(x\otimes y)$ into the positive part $p_\q$ and the negative part $\ell_{\q}$ as\vspace{-1mm}
\[
	\B(x\otimes y)= p_\q(x\otimes y)+ \s \ell_{\q}(x\otimes y), \qquad x\otimes y \in \HH,
\]
where\vspace{-2mm}
\begin{align}\label{rq}
&p_\q(x\otimes y)\eta = \sum_{k=1}^n \q^{n-k}\langle x, x_{\bar{k}}\rangle \langle y, x_k \rangle\, x_{\bar n }\otimes \cdots \otimes \check{x}_{\bar k} \otimes \cdots \otimes x_{\bar 1} \nonumber
\\[-1mm]
&\hphantom{p_\q(x\otimes y)\eta = \sum_{k=1}^n}
{}\otimes x_1\otimes \cdots \otimes \check{x}_k \otimes \cdots \otimes x_n,
\\
&\ell_{\q}(x\otimes y)\eta =\q^{n-1}\sum_{k=1}^n \q^{k-1}\langle x, x_{{k}}\rangle \langle y, x_{\bar{k}}\rangle\, x_{\bar n }\otimes \cdots \otimes \check{x}_{\bar k} \otimes \cdots \otimes x_{\bar 1} \nonumber
\\
&\hphantom{\ell_{\q}(x\otimes y)\eta =\q^{n-1}\sum_{k=1}^n}
\otimes x_1\otimes \cdots \otimes \check{x}_k \otimes \cdots \otimes x_n,\label{lq}
	\end{align}
	where $\eta=x_{\bar n }\otimes \cdots \otimes x_n$.
	We called the operator $p_\q$ the positive part and $\ell_{\q}$ the negative part because in the next section we see that they contribute to positive and negative partitions of type B, respectively.
\end{Remark}

The following commutation relation is almost the same as in \cite[Proposition 2.6]{BozejkoEjsmontHasebe2015}.
The difference is due to a \emph{twisted} inner product which appears in the second part of the following equation: $\langle x,\eta \rangle\langle y,\xi\rangle$.

\begin{Proposition}%\label{commutation}
	For $x\otimes y,\xi\otimes \eta \in \HH$, we have the commutation relation
	\begin{equation*}
		\B(x\otimes y)\B^\ast(\xi\otimes \eta )- \q \B^\ast(\xi\otimes \eta )\B(x\otimes y)= \langle x,\xi \rangle\langle y,\eta \rangle \id +\s \langle x,\eta \rangle\langle y,\xi\rangle (\q^2)^{N},
	\end{equation*}
where $(\q^2)^{N}$ is the operator on $\F$ defined by the linear extension of $(\q^2)^{N}\Omega \otimes \Omega =0$ and
	$\x\mapsto q^{2n} \x$ for $n\geq 1$.
\end{Proposition}

Now we come to calculate the norm of the creation operators.
The following theorem is inspired by~\cite[Theorem~2.9]{BozejkoEjsmontHasebe2015}. The proof is almost identical
to that of~\cite[Theorem~2.9]{BozejkoEjsmontHasebe2015}, and it can be omitted (the only difference is the estimate of the norm
at the end).

\begin{Theorem} Suppose that $x\otimes y \in \HH_{\R}$, $x \otimes y \neq0$.
\begin{enumerate}\itemsep=0pt
\item[$1.$]%\label{item1}
If $(\s,\q)\in A $, where $A=[0 ,1]\times (-1,0]$, then
		\begin{equation*}%\label{eq01}
			\|\B^\ast(x\otimes y)\|_{\s,\q}= \sqrt{\|x\|^2\|y\|^2 + \s \langle x,y\rangle^2}.
		\end{equation*}

		\item[$2.$]%\label{item2}
If $(\s,\q)\in B $, where $B= [-1 ,0)\times (-1,0]$, then
		\begin{equation*}%\label{eq02}
			\frac{\|x\|\|y\|}{\sqrt{1-\q}}\leq \|\B^\ast(x\otimes y)\|_{\s,\q}\leq \|x\|\|y\|.
		\end{equation*}

		\item[$3.$]%\label{item3}
If $(\s,\q)\in C $, where $C= \{(\s,\q)\mid |\s| \leq \q <1\}$, then
	\[
		\|\B^\ast(x\otimes y)\|_{\s,\q}= \frac{\|x\|\|y\|}{\sqrt{1-\q}}.
	\]

		\item[$4.$]%\label{item5}
Otherwise, if $(\s,\q)\in [-1 ,1]\times (-1,1)\setminus (A\cup B\cup C) $
	\[
		\frac{\|x\|\|y\|}{\sqrt{1-\q}}\leq \|\B^\ast(x\otimes y)\|_{\s,\q}\leq \sqrt{\frac{1+|\s|}{1-\q}}\|x\|\|y\|.
	\]
	\end{enumerate}
\end{Theorem}

\section{The double Gaussian operator of type B }
In this non-commutative setting, random variables are understood to be the elements of the
$\ast$-algebra generated by $\{\B(x\otimes y) ,\B^\ast(x\otimes y)\mid x\otimes y \in \HH_\R\}$. Particularly interesting are their mixed moments.
In order to work effectively on this object we need to define the corresponding statistics. We provide an explicit formula for the combinatorial moments, involving the number of crossings and negative pairs of a partition. First, we need to define the operators, the set of partitions and statistics of type B.

\subsection{The orthogonal polynomials}
\begin{Definition}
	The operator
	\begin{equation*}
		\G(x\otimes y)= \B(x\otimes y) +\B^\ast(x\otimes y),\qquad x\otimes y \in \HH_\R
	\end{equation*}
	on $\F$ is called the \emph{double Gaussian operator of type B}.
Denote by $\state$ the vacuum vector state $\state(\cdot)=\langle\Omega \otimes \Omega, \cdot\text{ } \Omega\otimes \Omega\rangle$.
\end{Definition}

Using a similar argument as in \cite[Theorem 3.3]{BozejkoEjsmontHasebe2015}, we can prove the following.

\begin{Theorem}
Suppose $\s,\q\in(-1,1)$ and $x\otimes y \in \HH_\R$, $\|x\|=\|y\|=1$. Let $\m$ be the probability distribution of $\G(x\otimes y)$ with respect to the vacuum state. Then $\m$ is equal to $\qMP_{\s\langle x,y\rangle^2,\q}$.
\end{Theorem}

\begin{proof}
	We observe that for $n \geq 1$, we have
	\begin{align*}
		\G(x\otimes y )x^{\otimes n}\otimes y^{\otimes n}
		&= \B^\ast(x\otimes y )x^{\otimes n}\otimes y^{\otimes n} +\B(x\otimes y) x^{\otimes n}\otimes y^{\otimes n}
\\
		&= x^{\otimes (n+1)}\otimes y^{\otimes (n+1)}+[n]_\q \big(1 + \s\langle x,y\rangle^2 \q^{n-1}\big)x^{\otimes (n-1)}\otimes y^{\otimes (n-1)},
	\end{align*}
	where Remark \ref{thm1} was used in the second line.
	Note that for $n=1$
\[
Q_{1}^{(\s\langle x,y\rangle^2,\q)}(\G(x\otimes y ))\Omega\otimes \Omega =\G(x\otimes y )\Omega \otimes \Omega =x\otimes y
\]
and by induction
\begin{gather*}
Q_{n+1}^{(\s\langle x,y\rangle^2,\q)}(\G(x\otimes y ))\Omega \otimes \Omega
\\ \qquad
{}=\G(x\otimes y ) Q_n^{(\s\langle x,y\rangle^2,\q)}( \G(x\otimes y ))\Omega\otimes \Omega
\\ \qquad\phantom{=}
{}-[n]_q \big(1 + \s\langle x, y\rangle^2 q^{n-1}\big)Q_{n-1}^{(\s\langle x,y\rangle^2,\q)}( \G(x\otimes y ))\Omega\otimes \Omega
\\ \qquad
{}=\G(x\otimes y ) x^{\otimes n}\otimes y^{\otimes n} - [n]_q\big(1+ \s\langle x, y\rangle^2 q^{n-1}\big)x^{\otimes (n-1)}\otimes y^{\otimes (n-1)}
\\ \qquad
{}=x^{\otimes n+1}\otimes y^{\otimes n+1}.
	\end{gather*}
	
Therefore, the map $\Phi\colon \bigl(\operatorname{span}\big\{x^{\otimes n}\otimes y^{\otimes n}\mid n \geq 0\big\}, \|\cdot\|_{\s,\q}\bigr) \to L^2(\R,\qMP_{\s\langle x,y\rangle^2 ,\q})$ defined by
\[
\Phi\big(x^{\otimes n}\otimes y^{\otimes n}\big)= Q_n^{(\s\langle x,y\rangle^2,\q)}(t)
\]
is an isometry.
	Since $\Phi$ is an isometry, we get
	$\langle \Omega \otimes \Omega, \G(x\otimes y)^n\Omega\rangle_{\s,\q} = m_n(\qMP_{\s\langle x,y\rangle^2,\q})$, where $m_n(\mu)$ is the $n$-th moment of measure $\mu$.
	Since $\qMP_{\s\langle x,y\rangle^2,\q}$ is a compactly supported measure, the Hamburger moment problem has a unique solution and hence $\qMP_{\s\langle x,y\rangle^2,\q}=\m$.
\end{proof}

\subsection{Pair partitions of type B}
%\label{subsec:PartitionsB}

Let $S$ be an ordered set. Then $\pi = \{ V_1,\dots, V_p\}$ is a partition of $S$, if the $V_i \neq \varnothing$ are
ordered and disjoint sets $V_i=(v_1,\dots,v_k)$, where $v_1<\dots<v_k$, whose union is $S$.
For $V \in \pi$ , we say that~$V$ is a \emph{block of $\pi$}.
Any partition $\pi$ defines an equivalence relation on $S$,
denoted by $\sim_\pi$, such that the equivalence classes are the
blocks of~$\pi$.
Therefore, $i\sim_\pi j$ if $i$ and $j$ belong to the same block of $\pi$.
A block of $\pi$ is called a \emph{singleton} if it consists of one element.
Similarly, a~block of $\pi$ is called a \emph{pair} if it consists of two elements.
Let $\Sing(\pi)$ and $\Pair(\pi)$ denote the set of all singletons and pairs of $\pi$, respectively.
The set of partitions of $S$ is denoted by $\Part(S)$, in the case where $S =
[n] := \{1, \dots , n\}$ we write $\Part(n)$ := $\Part([n])$.
We denote by $\P_{1,2}(S)$ the subset of partitions $\pi \in \Part(S)$ whose every block is either a
pair or a singleton.

From now on, we will work on a set $[\pm n]:=\{\bar n, \dots , \bar 1, 1,\dots, n\}$.
For a pair $V=(a,b)$ (or a~singleton $V=(a)$), we denote its reflection by $\overline V=(\bar b , \bar a)$ ($\overline V=( \bar a)$), where $\bar a =-a$.
Similarly, we define
\[
\bar \pi: =\big\{\overline{V}\mid V\in \pi,\text{ } \pi \in \P_{1,2}([\pm n])\big\}.
\]

\begin{Definition} \label{def:partycji}
	We denote by $\P_{1,2}^{B}(n)$ the subset of partitions $\pi \in \P_{1,2}([\pm n])$
	such that they are symmetric $\overline{\pi} =
	\pi$ (which is invariant under the bar operation), but every pair $V \in \pi$ is different from its reflection
	$\overline{V}$, i.e.,~$V \neq \overline{V}$.
	We call $\P_{1,2}^{B}(n)$ the set of partitions of type B.
\end{Definition}

From Definition~\ref{def:partycji} it follows that for every block in $B\in \pi$, $\pi \in \P_{1,2}^{B}(n)$ there exists a unique {reflection block} $\bar B\in\pi$.
This leads to one more definition.
We call $\BL= ((\bar b, \bar a),( a,b))$ a \emph{B-pair} if $\bar b <b$, $|{a}| <|b|$ and $(\bar b,\bar a ), (a,b)\in \pi$.
The B-pair $((\bar b, \bar a),( a,b))$ is called \emph{positive} if $b\times a >0 $; otherwise
it is called \emph{negative}; see Figure~\ref{fig:coneced}.
The set of such B-pairs is denoted by $\Pair_B(\pi)$.
Let us notice that from Definition~\ref{def:partycji} it follows that
the element $s$ is a singleton of $\pi \in \P_{1,2}^{B}(n)$ if and only if $\bar s$ is also a singleton of $\pi$, thus we can define the subset of B-singletons $((\bar a),(a))$ as
\[
\Sing_B(\pi):=\{((\bar a), (a)) \mid \bar a,a \in \Sing(\pi), \,\bar a < a\}, \qquad \pi \in \P_{1,2}^{B}(n).
\]
Let
\[
\P_{2}^{B}(n):=\big\{\pi \in \P_{1,2}^{B}(n)\mid \Sing_B(\pi)=\varnothing\big\}
\]
and
\[
\PA_2(n):=\big\{\pi \in \P_{2}^{B}(n)\mid (\bar b,\bar a ),\, (a,b)\in \pi \implies ((\bar b,\bar a ), (a,b))\text{ is positive}\big\},
\]
i.e., it is a subset of $\P_{2}^{B}(n)$, with only positive B-pairs.

\begin{Remark}\quad
\begin{enumerate}
\item [$1.$] We note that a B-pair is not a block of $\pi$, but a pair of reflection pairs.
	
\item [$2.$] We note that $\#\PA_2(n)=(n-1)!!$ for even $n$, which is the same as the number of the classical pair partitions on $n$ points.
	In Figure~\ref{figexample1}, they are depicted on the top row.
\end{enumerate}	
\end{Remark}

\begin{figure}[h]
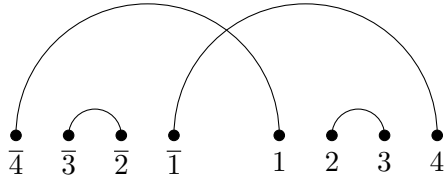
\centering
\MatchingMeandersabc{4}{-4/1, -3/-2,2/3,-1/4 }{}
\caption{The example of $\pi \in \P_{2}^{B}(4)$ with a positive B-pair $((\bar 3,\bar 2 ), (2,3))$ and a negative B-pair $((\bar 4,1), (\bar 1,4))$. }
	\label{fig:coneced}
\end{figure}

We introduce some partition statistics.
For two pairs $V$ and $ W$, we introduce the relation $\text{cr}$ as follows:
\begin{gather*}
V\stackrel{\text{cr}}{\sim}W \iff
V=(i,j),\qquad W=(k,l)\quad \text{such that}\quad i<k<j<l,
\end{gather*}
For a set partition $\pi\in \P_{1,2}^{B}(n)$, let $\Cr(\pi)$ be the number of crossings of B-pairs, i.e.,
\begin{align*}
\Cr(\pi)={}&\#\big\{\big(\big(\overline{V}, V \big),\big(\overline{ W}, W\big)\big)\in \Pair_B(\pi)\times \Pair_B(\pi) \mid \text{$V\stackrel{\text{cr}}{\sim}W $}\big\}
\\
&+\#\big\{\big(\big(\overline{V}, V\big),\big(\overline{ W}, W\big)\big)\in \Pair_B(\pi)\times \Pair_B(\pi) \mid \text{$\overline{V}\stackrel{\text{cr}}{\sim}W $}\big\}.
\end{align*}
For two blocks $V$, $W$ of a set partition, we say that $W$ \emph{covers} $V$ if there are $i,j \in W$ such that $i <k <j$ for any $k\in V$ and then we write $V\stackrel{\text{cs}}{\sim}W $. For $\pi\in\P_{1,2}^{B}(n)$, let $\InS(\pi)$
be the number of pairs of a B-singleton and a covering B-pair:
\begin{align*}
\begin{split}
\InS(\pi)&=\#\big\{\big(\big(\overline{ V}, V \big),\big(\overline{ W}, W\big)
\big) \in \Sing_B(\pi) \times \Pair_B(\pi) \mid V\stackrel{\text{cs}}{\sim}W
\big\}
\\
 & +\#\big\{\big(\big(\overline{V}, V \big),\big(\overline{W}, W\big)\big) \in \Sing_B(\pi) \times \Pair_B(\pi) \mid \overline{V}\stackrel{\text{cs}}{\sim}W\big\}.
 \end{split}
\end{align*}
We say that the B-pair $(\overline{ W} , W)$ cover the B-singleton $\big(\overline{V}, V\big)$ if $V\stackrel{\text{cs}}{\sim}{W} $ or $\overline{V}\stackrel{\text{cs}}{\sim}{ W} $ (equivalently, $\big(\overline{V}, V\big)$ is the inner singleton of $\big(\overline{W} , W\big)$).

Let $\NB(\pi)$ be the number of negative B-pairs of $\Pair_B(\pi)$, where $\pi \in \P_{1,2}^{B}(n)$.
A set partition $\pi \in \P_{1,2}^{B}(n)$ is \emph{non-crossing} if $\Cr(\pi)=0$. The set of non-crossing pair partitions of $[\pm n]$ is denoted by $\NC^B_2(n)$, and by $\NC^A_2(n)$ we denote the subset of $\NC^B_2(n)$ where all B-pairs are positive.

\begin{Remark} \label{jakliczyc}
While reading this, the first impression seems to be that the procedure of counting the crossings is complicated. This is not true since it can be read from the figure as follows. First, we draw a vertical line in the center as in Figure~\ref{fig:examplejakliczyccrosingi}.
Then we count only the crossings on the left of this vertical line.
We count the number of
negative B-pairs as the number of B-pairs crossed by a vertical line ($\scriptstyle{\blacklozenge}$ points on Figure~\ref{fig:examplejakliczyccrosingi}).
Similarly we count the pairs of the form
\[
(\text{B-singleton, covering B-pair})
\]
as the number of such a pairs with at least one leg on the left of this vertical line.

\begin{figure}[htp]
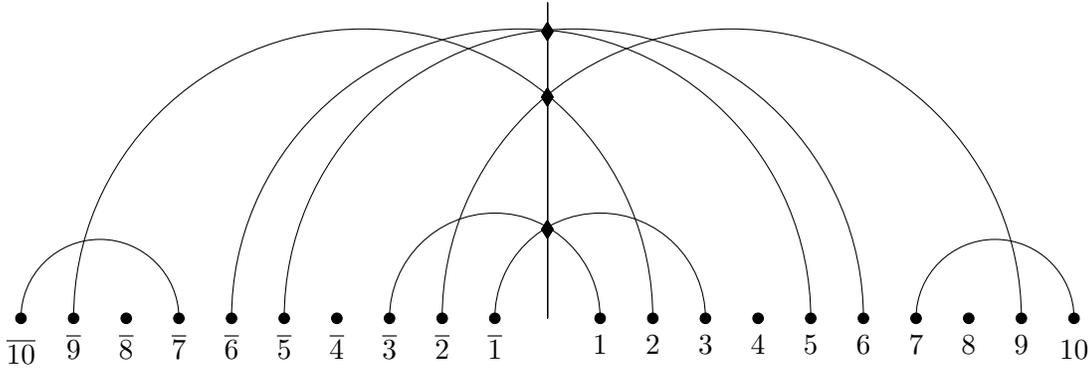

\centering
\MatchingMeanders{10}{-1/3, -3/1 ,-6/5,-5/6,7/10,-10/-7,-2/9,-9/2}
\caption{The example of statistic of partition $\pi \in \P_{1,2}^{B}(10)$, i.e., $\Cr(\pi)=4$, $\NB(\pi)=3$, $\InS(\pi)=5$.}
\label{fig:examplejakliczyccrosingi}
\end{figure}
\end{Remark}

Now we prove the following theorem, which shows the relationship between the set of partitions of type B
(corresponding statistic) and a joint action of Gaussian operators on a vacuum vector.

\begin{Theorem}\label{lem101}
For any $i\in\{1,\ldots,2n\}$ and $x_{\bar i}\otimes x_i \in \HH_\R$, we have
\begin{equation}\label{formula101}
	\state(\G(x_{\overline{ 2n}}\otimes x_{ 2n})\cdots \G(x_{\bar 1}\otimes x_{ 1}))= \sum_{\pi\in\PB_{2}(2n)} \s^{\NB(\pi)}\q^{\Cr(\pi)} \prod_{\substack{(i,j) \in \Pair(\pi)} }\langle x_i, x_j\rangle.
\end{equation}
\end{Theorem}

\begin{proof}
Given
$\epsilon=(\epsilon(1), \dots, \epsilon(n))\in\{1,\ast\}^{n}$, let
$\P_{1,2;\epsilon}^{B}(n)$ be the set of partitions $\pi\in
\P_{1,2}^{B}(n)$ such that
\begin{itemize}\itemsep=0pt
\item if $((\bar b ,\bar a) , (a,b))$ is a B-pair of
$\Pair_B(\pi)$, then $\epsilon(|a|)= \ast$, $\epsilon(|b|)=1$,
\item if $\{c\}$ is a singleton in $\pi$, then $\epsilon( |c|)=\ast$.
\end{itemize}
We will prove first that
\begin{gather}
\B^{\epsilon(n)}(x_{\overline{n}}\otimes x_{n})\cdots \B^{\epsilon(1)}(x_{\overline{1}}\otimes x_1)\Omega \otimes \Omega \nonumber
\\
\qquad{} = \sum_{\pi\in\PB_{1,2;\epsilon}(n)} \s^{\NB(\pi)}q^{\Cr(\pi)+ \InS(\pi)}
\prod_{\substack{(i,j) \in \Pair(\pi)} }\langle x_i, x_j\rangle
\bigotimes_{i\in \Sing(\pi)} x_i\label{formula10112}
\end{gather}
holds, where
\[
\Sing(\pi) = \{(\bar v_m),\dots,(\bar v_1), (v_1),\dots,(v_m)\} \subset\bar{\N} \cup \N, \qquad
\bar v_m<\cdots < \bar v_1< v_1<\cdots < v_m,
\]
and
$\bigotimes_{i\in \Sing(\pi)} x_i$ denotes the tensor product $x_{\bar v_m}\otimes \cdots \otimes x_{\bar v_1} \otimes x_{v_1}\otimes \cdots \otimes x_{v_m}$.
If
\[
\#\{i\in[j] \mid \epsilon(i)=1\}>\#\{i\in[j] \mid \epsilon(i)=\ast\}
\]
for some $j\in\{1,\dots,n\}$, then we understand the sum over the empty set is 0 since $\PB_{1,2;\epsilon}(n)=\varnothing$ in this case.
The proof of \eqref{formula10112} is given by induction.

When $n=1$, then $\B(x_{\bar 1}\otimes x_1)\Omega\otimes \Omega=0$ and $\B^\ast(x_{\bar 1}\otimes x_1)\Omega\otimes \Omega=x_{\bar 1}\otimes x_1$ and hence the formula \eqref{formula10112} is true.

Suppose that the \eqref{formula10112} is true for $n = k$.
We will show that the action of $\B^{\epsilon(k+1)}(x_{\,\overline{k+1}}{}\otimes x_{k+1})$ for $\epsilon(k+1)=\ast$ or $\epsilon(k+1)=1$ corresponds to the inductive pictorial description of set partitions of type B.

\smallskip\noindent
\emph{Case} 1. If $\epsilon(k+1)=\ast$, then the operator $\B^\ast(x_{\,\overline{k+1}}\otimes x_{k+1})$ acts on the tensor product, putting~$x_{k+1}$ on the right and $x_{\,\overline{k+1}}$ on the left. This operation pictorially corresponds to adding the singleton $(\overline{k+1}
)$ and $(k+1)$ to $\pi\in\PB_{1,2;\epsilon}(k)$, to yield the new type-B partition $\tilde{\pi}\in\PB_{1,2;\epsilon}(k+1)$ such that $((\overline{k+1}), (k+1))\in \Sing_B(\tilde{\pi})$. This map $\pi\mapsto \tilde{\pi} $ does not change the numbers $\NB$, $\Cr$ or $\InS$, which is compatible with the fact that the action of $\B^\ast(x_{\,\overline{k+1}}\otimes x_{k+1})$ does not change the coefficient.
Hence the formula~\eqref{formula10112} holds when $n=k+1$ and $\epsilon(k+1)=\ast$.

\smallskip\noindent
\emph{Case} 2. If $\epsilon(k+1)=1$, then we have two subcases, depending on $p_\q$ and $\ell_{\q}$.

If the positive part $p_\q $ (subcase (a)) (resp. $\s \ell_{\q}$ -- subcase (b)) acts on the tensor product on the right hand side of \eqref{formula10112} (for fixed $\pi$), then new $m$ terms appear by using \eqref{rq}. We obtain the equations
\begin{subequations}
\begin{align}
	\begin{split}
		p_\q(x_{\,\overline{k+1}}\otimes x_{k+1}) & x_{\bar v_m}\otimes \cdots \otimes x_{v_m}=
		\sum_{i=1}^m \q^{m-i}\langle x_{\,\overline{k+1}},x_{\bar v_i}\rangle \langle x_{v_i} ,x_{k+1}\rangle\, \\ &\times x_{\bar v_m }\otimes \cdots \otimes \check{x}_{\bar v_i} \otimes \cdots \otimes x_{\bar v_1} \otimes x_{v_1}\otimes \cdots \otimes \check{x}_{v_i} \otimes \cdots \otimes x_m,
	\end{split}\label{eq:equatinpomocdowodgolwny1}
	\\
	\begin{split}
		\alpha\ell_\q(x_{\,\overline{k+1}}\otimes x_{k+1}) & x_{\bar v_m}\otimes \cdots \otimes x_{v_m}=
		\alpha\q^{m-1}\sum_{i=1}^m \q^{i-1}
		\langle x_{\,\overline{k+1}},x_{ v_i}\rangle \langle x_{\bar v_i} ,x_{k+1}\rangle\, \\& \times x_{\bar v_m }\otimes \cdots \otimes \check{x}_{\bar v_i} \otimes \cdots \otimes x_{\bar v_1} \otimes x_{v_1}\otimes \cdots \otimes \check{x}_{v_i} \otimes \cdots \otimes x_m.
	\end{split}\label{eq:equatinpomocdowodgolwny2}
\end{align}
\end{subequations}
Now we will focus on the $i$-th summand of the above equations.

We fix $\pi\in \PB_{1,2;\epsilon}(k)$ and suppose that $\pi$ has singletons
\begin{align*}
\bar v_m<\cdots <\bar v_{ i} < \cdots <\bar v_1 < v_1<\cdots <v_i <\cdots <v_m,\text{ where } i\in [m].
\end{align*}

We assume that $\Pair_B(\pi)$ contains the B-pairs $\big(\overline{ W}_1 , W_1\big),\dots, \big(\overline{ W}_u , W_u\big)$ which cover the B-singleton $((\bar v_i),( v_i))$.
Case~2(a)~-- see Figure~\ref{figtracepositivea}.
In the $i$-th term of \eqref{eq:equatinpomocdowodgolwny1} the inner product $\langle x_{\,\overline{k+1}},x_{\bar v_i}\rangle \langle x_{v_i} ,x_{k+1}\rangle$ appears with coefficient $\q^{m-i}$. Pictorially this corresponds to getting a set partition $\tilde{\pi} \in \PB_{1,2;\epsilon}(k+1)$ by adding the positive B-pair $\big((\overline{k+1}, \overline{v}_i) , (v_i, k+1)\big)$ to $\Pair_B(\tilde \pi)$. This new B-pair crosses the B-pairs $\big(\overline{W}_j , W_j\big)$, $j\in\{1,\dots,u\}$ and so increases the crossing number by~$u$ but decreases the number $\InS(\pi)$ by $u$ because originally $((\bar v_i),( v_i))$ was the inner singleton of B-pairs $\big(\overline{W}_j , W_j\big)$, $j\in\{1,\dots,u\}$.
Now the new inner B-singletons $((\bar v_{i+1}), (v_{i+1})), \dots,(( \bar v_m), (v_m))$ of B-pair $((\overline{k+1}, \overline{v}_i) , (v_i, k+1))$
appear.
Altogether we have
\[
\Cr(\tilde{\pi})=\Cr(\pi)+u, \qquad
\InS(\tilde{\pi})=\InS(\pi)+m-i-u\qquad \text{and}\qquad
\NB(\tilde{\pi})=\NB(\pi).
\]
So the exponent of $\q$ increases by $m-i$. This factor $\q^{m-i}$ is exactly the factor appearing in equation~\eqref{eq:equatinpomocdowodgolwny1}.

\begin{figure}[h]
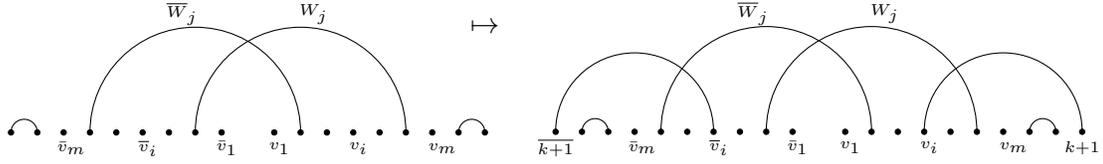

\centering
\MatchingMeandersPositiveTracea{9}{-9/-8,8/9,-6/2,-2/6 }{}
\MatchingMeandersPositiveTraceb{10}{-9/-8,8/9,-6/2,-2/6 ,-10/-4,4/10 }{}
\caption{The visualization of the action of $p_\q(x_{\,\overline{k+1}}\otimes x_{k+1})$.}
\label{figtracepositivea}
\end{figure}

Case 2(b) -- see Figure~\ref{figtracepositive}.
In the $i$-th term of \eqref{eq:equatinpomocdowodgolwny2} the inner product $\langle x_{\,\overline{k+1}},x_{ v_i}\rangle \langle x_{\bar v_i} ,x_{k+1}\rangle$ appears with the coefficient $\s \q^{m+ i-2}$.
Graphically this
corresponds to getting a set partition $\tilde{\pi} \in \PB_{1,2;\epsilon}(k+1)$ by adding $\overline{k + 1}$ and $k + 1$ to $\pi$ and creating the new negative B-pair $\big((\overline{k+1}, v_{i}) , (\overline v_ i, k+1)\big)\in \Pair_B(\tilde \pi)$.
Similarly to Case 2(a), we count the change of numbers and get
\[
\Cr(\tilde{\pi})=\Cr(\pi)+u,\quad
\InS(\tilde{\pi})=\InS(\pi)+m-1+i-1-u\quad\text{and}\quad \NB(\tilde{\pi})=\NB(\pi)+1.
\]
Altogether, when moving from $\pi$ to $\tilde{\pi}$, the exponent of $\s$ increases by $1$ and the exponent of $\q$ increases by ${m+ i-2}$, which coincides with the coefficient appearing in the action of $\s \ell_{\q}(x_{\,\overline{k+1}}\otimes x_{k+1})$, creating the inner product $\langle x_{\,\overline{k+1}},x_{ v_i}\rangle \langle x_{\bar v_i}, x_{k+1}\rangle$.

\begin{figure}[h]
%\begin{center}
%\QQ{$(b)$}
%\end{center}%
\centering
\MatchingMeandersPositiveTracea{9}{-9/-8,8/9,-6/2,-2/6 }{} % $\xrightarrow{\text{trace action }}$
\MatchingMeandersPositiveTraceb{10}{-9/-8,8/9,-6/2,-2/6 ,-10/4,-4/10 }{}
\caption{The visualization of the action of $\s \ell_{\q}(x_{\,\overline{k+1}}\otimes x_{k+1})$.}
\label{figtracepositive}
\end{figure}

Note that as $\pi$ runs over $\PB_{1,2;(\epsilon(1),\dots,\epsilon(k))}(k)$, every set partition $\tilde{\pi}\in \PB_{1,2;(\epsilon(1),\dots,\epsilon(k),1)}(k+1)$ appears exactly once, either in Case~2(a) or in Case~2(b). Therefore, in Case~2, the pictorial inductive step and the actual action of $\B(x_{\,\overline{k+1}}\otimes x_{k+1})$ both create the same terms with the same coefficients, and hence the formula \eqref{formula10112} is true when $n=k+1$ and $\epsilon(k+1)=1$. Case~1 and Case~2 show by induction that the formula \eqref{formula10112} holds for all $n \in \N$.

For a given set of B-pairs $\{((\bar b_1, \bar a_1),(a_1,b_1)), \dots, ((\bar b_{n/2}, \bar a_{n/2}),(a_{n/2},b_{n/2}))\}\in \Pair_B( \pi) $,
where $\pi \in \PB_{2}(n)$, we denote the set of \emph{left and right legs} of $\pi$ by $l_\pi :=\{ a_1,\bar a_1,\dots, a_{n/2}, \bar a_{n/2} \}$, and $r_\pi :=\{ b_1,\bar b_1,\dots, b_{n/2}, \bar b_{n/2} \}$. Formula \eqref{formula101} follows from \eqref{formula10112} by taking the sum over all $\epsilon$ such that $\Sing(\pi)=\varnothing$. In this case, we understand that $\bigotimes_{i\in \Sing(\pi)} x_i=\Omega\otimes \Omega $ and
\begin{align*}
\bigsqcup_{\e\in\{1,\ast\}^{n}} \PB_{2;\e}(n)&=\bigsqcup_{\e\in\{1,\ast\}^{n}} \big\{\pi \in \P_{2}^{B}(n)\mid (\bar b,\bar a ), (a,b)\in \pi \implies \epsilon(|a|)= \ast,\, \epsilon(|b|)=1\big\}
\\
& =\bigsqcup_{\substack{ L\subset [\pm n] \\ \# L =n
}} \big\{\pi \in \P_{2}^{B}(n)\mid
l_\pi=L \text{ and }r_\pi = [\pm n]\setminus L\big\}
\\
&=\PB_{2}(n).
\end{align*}

 Finally, applying the state action, we get
\begin{align*}
&\state(\G(x_{\overline{ 2n}}\otimes x_{ 2n})\cdots \G(x_{\bar 1}\otimes x_{ 1}))
\\
&\qquad =\sum_{\e\in\{1,\ast\}^{2n}} \state\big( \B^{\epsilon(2n)}(x_{\overline{2n}}\otimes x_{2n})\cdots \B^{\epsilon(1)}(x_{\overline{1}}\otimes x_1)\big)
\\
&\qquad =\sum_{\e\in\{1,\ast\}^{2n}} \sum_{\substack{\pi\in\PB_{1,2;\epsilon}(2n)\\ \Sing(\pi)=\varnothing }} \s^{\NB(\pi)}q^{\Cr(\pi)+ \InS(\pi)}
\prod_{\substack{(i,j) \in \Pair(\pi)} }\langle x_i, x_j\rangle
\\
&\qquad =\sum_{\e\in\{1,\ast\}^{2n}} \sum_{\substack{\pi\in\PB_{2;\epsilon}(2n) }} \s^{\NB(\pi)}q^{\Cr(\pi)}
\prod_{\substack{(i,j) \in \Pair(\pi)} }\langle x_i, x_j\rangle
\\
&\qquad = \sum_{\pi\in\PB_{2}(2n)} \s^{\NB(\pi)}\q^{\Cr(\pi)} \prod_{\substack{(i,j) \in \Pair(\pi)} }\langle x_i, x_j\rangle.
\tag*{\qed}
\end{align*}
\renewcommand{\qed}{}
\end{proof}

\begin{Example} \label{exkumulanty}The set of partitions of type B for \[\state(\G(x_{\overline{ 4}}\otimes x_{ 4})\cdots \G(x_{\bar 1}\otimes x_{ 1}))\]
can be graphically represented as shown in Figure~\ref{figexample1}.

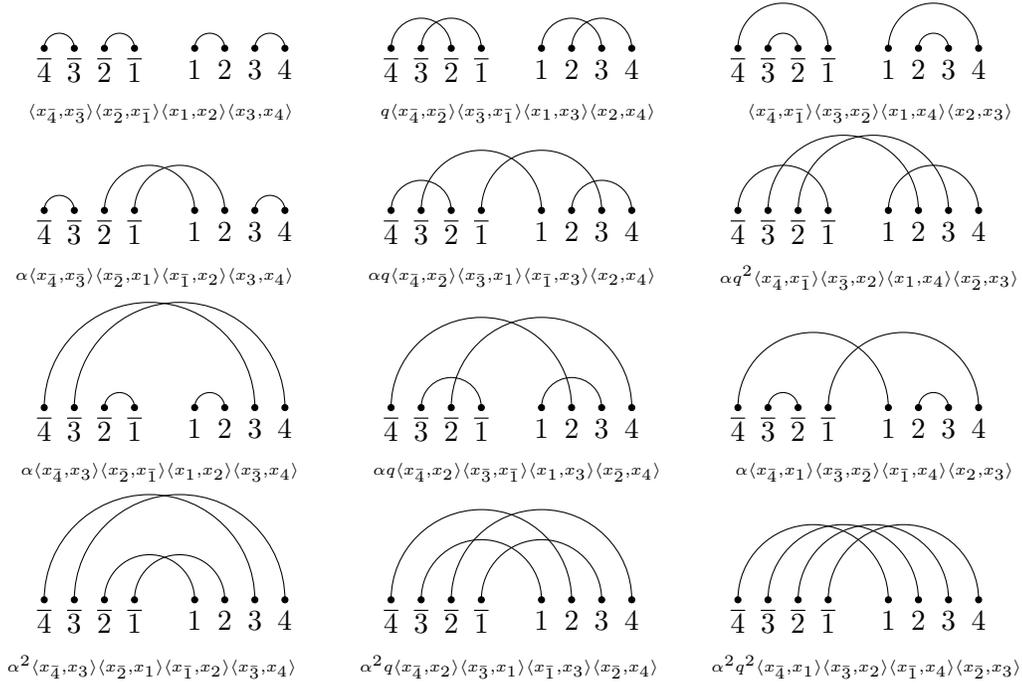
\begin{figure}[htp]
\centering
\MatchingMeandersab{4}{-4/-3, -2/-1,1/2,3/4 }{} \hspace{1.5em} \MatchingMeandersab{4}{-4/-2, -3/-1,1/3,2/4 }{} \hspace{1.5em} \MatchingMeandersab{4}{-4/-1, -3/-2,2/3,1/4 }{}
\begin{tikzpicture}[scale=0.5]
\node at (0,1) {};
\node at (4,1) {$\scriptscriptstyle{\langle x_{\bar 4}, x_{\bar 3}\rangle \langle x_{\bar 2}, x_{\bar 1}\rangle \langle x_{1}, x_{2}\rangle\langle x_{3}, x_{4}\rangle }$};
\node at (13.5,1) {$\scriptscriptstyle{\q \langle x_{\bar 4}, x_{\bar 2}\rangle \langle x_{\bar 3}, x_{\bar 1}\rangle \langle x_{1}, x_{3}\rangle\langle x_{2}, x_{4}\rangle }$};
\node at (23,1) {$\scriptscriptstyle{\ \langle x_{\bar 4}, x_{\bar 1}\rangle \langle x_{\bar 3}, x_{\bar 2}\rangle \langle x_{1}, x_{4}\rangle\langle x_{2}, x_{3}\rangle }$};
\end{tikzpicture}
\centering
\MatchingMeandersab{4}{-4/-3, -2/1,-1/2,3/4 }{} \hspace{1.5em} \MatchingMeandersab{4}{-4/-2, -3/1,-1/3,2/4 }{}
\hspace{1.5em} \MatchingMeandersab{4}{-4/-1, -3/2,-2/3,1/4 }{}
\begin{tikzpicture}[scale=0.5]
\node at (0,1) {};
\node at (4,1) {$ \scriptscriptstyle{\s\langle x_{\bar 4}, x_{\bar 3}\rangle \langle x_{\bar 2}, x_{ 1}\rangle \langle x_{\bar 1}, x_{2}\rangle\langle x_{3}, x_{4}\rangle }$};
\node at (13.5,1) {$\scriptscriptstyle{\s \q \langle x_{\bar 4}, x_{\bar 2}\rangle \langle x_{\bar 3}, x_{1}\rangle \langle x_{\bar 1}, x_{3}\rangle\langle x_{2}, x_{4}\rangle }$};
\node at (23,1) {$\scriptscriptstyle{\alpha\q^2 \langle x_{\bar 4}, x_{\bar 1}\rangle \langle x_{\bar 3}, x_{ 2}\rangle \langle x_{1}, x_{4}\rangle\langle x_{\bar 2}, x_{3}\rangle }$};
\end{tikzpicture}
\centering
\MatchingMeandersab{4}{-4/3, -2/-1,1/2,-3/4 }{} \hspace{1.5em} \MatchingMeandersab{4}{-4/2, -3/-1,1/3,-2/4 }{} \hspace{1.5em} \MatchingMeandersab{4}{-4/1, -3/-2,2/3,-1/4 }{}
\begin{tikzpicture}[scale=0.5]
\node at (0,1) {};
\node at (4,1) {$ \scriptscriptstyle{\s\langle x_{\bar 4}, x_{ 3}\rangle \langle x_{\bar 2}, x_{ \bar 1}\rangle \langle x_{ 1}, x_{ 2}\rangle\langle x_{\bar 3}, x_{4}\rangle }$};
\node at (13.5,1) {$\scriptscriptstyle{\s\q \langle x_{\bar 4}, x_{ 2}\rangle \langle x_{\bar 3}, x_{\bar 1}\rangle \langle x_{ 1}, x_{3}\rangle\langle x_{\bar 2}, x_{4}\rangle }$};
\node at (23,1) {$\scriptscriptstyle{\s\langle x_{\bar 4}, x_{ 1}\rangle \langle x_{\bar 3}, x_{ \bar 2}\rangle \langle x_{\bar 1}, x_{4}\rangle\langle x_{ 2}, x_{3}\rangle }$};
\end{tikzpicture}
\centering
\MatchingMeandersab{4}{-4/3, -2/1,-1/2,-3/4 }{} \hspace{1.5em} \MatchingMeandersab{4}{-4/2, -3/1,-1/3,-2/4 }{}
\hspace{1.5em} \MatchingMeandersab{4}{-4/1, -3/2,-2/3,-1/4 }{}
\begin{tikzpicture}[scale=0.5]
\node at (0,1) {};
\node at (4,1) {$ \scriptscriptstyle{\s^2\langle x_{\bar 4}, x_{ 3}\rangle \langle x_{\bar 2}, x_{ 1}\rangle \langle x_{ \bar 1}, x_{ 2}\rangle\langle x_{\bar 3}, x_{4}\rangle }$};
\node at (13.5,1) {$\scriptscriptstyle{\s^2\q \langle x_{\bar 4}, x_{ 2}\rangle \langle x_{\bar 3}, x_{1}\rangle \langle x_{\bar 1}, x_{3}\rangle\langle x_{\bar 2}, x_{4}\rangle }$};
\node at (23,1) {$\scriptscriptstyle{\s^2q^2\langle x_{\bar 4}, x_{ 1}\rangle \langle x_{\bar 3}, x_{ 2}\rangle \langle x_{\bar 1}, x_{4}\rangle\langle x_{ \bar 2}, x_{3}\rangle }$};
\end{tikzpicture}
\caption{The statistics and the set of partitions of type B for the fourth moment. }
\label{figexample1}
\end{figure}

\end{Example}
Let $\A$ be the von Neumann algebra generated by $\{\G(x\otimes y)\mid x\otimes y \in \HH_\R\}$ for $|q|<1$.

\begin{Proposition}
Suppose that $\dim(H_{\R}) \geq 2$ and $|q|<1$.
From Example~$\ref{exkumulanty}$, we can easily calculate that the vacuum state is a trace on $\A$ if and only if $\s=0$.
\end{Proposition}

\begin{proof}[Proof of the last statement]
By using Example \ref{exkumulanty}, we obtain
\begin{align*}
\begin{split}
&\state(\G(x_{\overline{ 4}}\otimes x_{ 4})\G(x_{\overline{ 3}}\otimes x_{ 3})\G(x_{\overline{ 3}}\otimes x_{ 2}) \G(x_{\bar 1}\otimes x_{ 1}))
\\
&\qquad=\langle x_{\bar 4}, x_{\bar 3}\rangle \langle x_{\bar 2}, x_{\bar 1}\rangle \langle x_{1}, x_{2}\rangle\langle x_{3}, x_{4}\rangle +\q \langle x_{\bar 4}, x_{\bar 2}\rangle \langle x_{\bar 3}, x_{\bar 1}\rangle \langle x_{1}, x_{3}\rangle\langle x_{2}, x_{4}\rangle
\\
&\qquad\phantom{=}+ \langle x_{\bar 4}, x_{\bar 1}\rangle \langle x_{\bar 3}, x_{\bar 2}\rangle \langle x_{1}, x_{4}\rangle\langle x_{2}, x_{3}\rangle +\s\langle x_{\bar 4}, x_{\bar 3}\rangle \langle x_{\bar 2}, x_{ 1}\rangle \langle x_{\bar 1}, x_{2}\rangle\langle x_{3}, x_{4}\rangle
\\
&\qquad\phantom{=}+\s \q \langle x_{\bar 4}, x_{\bar 2}\rangle \langle x_{\bar 3}, x_{1}\rangle \langle x_{\bar 1}, x_{3}\rangle\langle x_{2}, x_{4}\rangle +\alpha\q^2 \langle x_{\bar 4}, x_{\bar 1}\rangle \langle x_{\bar 3}, x_{ 2}\rangle \langle x_{1}, x_{4}\rangle\langle x_{\bar 2}, x_{3}\rangle
\\
&\qquad\phantom{=}+ \s\langle x_{\bar 4}, x_{ 3}\rangle \langle x_{\bar 2}, x_{ \bar 1}\rangle \langle x_{ 1}, x_{ 2}\rangle\langle x_{\bar 3}, x_{4}\rangle +\s\q \langle x_{\bar 4}, x_{ 2}\rangle \langle x_{\bar 3}, x_{\bar 1}\rangle \langle x_{ 1}, x_{3}\rangle\langle x_{\bar 2}, x_{4}\rangle
\\
&\qquad\phantom{=}+\s\langle x_{\bar 4}, x_{ 1}\rangle \langle x_{\bar 3}, x_{ \bar 2}\rangle \langle x_{\bar 1}, x_{4}\rangle\langle x_{ 2}, x_{3}\rangle +\s^2\langle x_{\bar 4}, x_{ 3}\rangle \langle x_{\bar 2}, x_{ 1}\rangle \langle x_{ \bar 1}, x_{ 2}\rangle\langle x_{\bar 3}, x_{4}\rangle
\\
&\qquad\phantom{=}+\s^2\q \langle x_{\bar 4}, x_{ 2}\rangle \langle x_{\bar 3}, x_{1}\rangle \langle x_{\bar 1}, x_{3}\rangle\langle x_{\bar 2}, x_{4}\rangle +\s^2q^2\langle x_{\bar 4}, x_{ 1}\rangle \langle x_{\bar 3}, x_{ 2}\rangle \langle x_{\bar 1}, x_{4}\rangle\langle x_{ \bar 2}, x_{3}\rangle,
\end{split}
\end{align*}
and by permuting $x_{\overline{ 4}}\otimes x_{ 4},\dots, x_{\bar 1}\otimes x_{ 1}$,
\begin{align*}
&\state(\G(x_{\overline{ 3}}\otimes x_{ 3})\G(x_{\overline{ 3}}\otimes x_{ 2}) \G(x_{\bar 1}\otimes x_{ 1})\G(x_{\bar 4}\otimes x_{ 4}))
\\
&\qquad=\langle x_{\bar 3}, x_{\bar 2}\rangle \langle x_{\bar 1}, x_{\bar 4}\rangle \langle x_{4}, x_{1}\rangle\langle x_{2}, x_{3}\rangle
+\q \langle x_{\bar 3}, x_{\bar 1}\rangle \langle x_{\bar 2}, x_{\bar 4}\rangle \langle x_{4}, x_{2}\rangle\langle x_{1}, x_{3}\rangle
\\
&\qquad\phantom{=}+ \langle x_{\bar 3}, x_{\bar 4}\rangle \langle x_{\bar 2}, x_{\bar 1}\rangle \langle x_{4}, x_{3}\rangle\langle x_{1}, x_{2}\rangle
+\s\langle x_{\bar 3}, x_{\bar 2}\rangle \langle x_{\bar 1}, x_{ 4}\rangle \langle x_{\bar 4}, x_{1}\rangle\langle x_{2}, x_{3}\rangle
\\
&\qquad\phantom{=}+\s \q \langle x_{\bar 3}, x_{\bar 1}\rangle \langle x_{\bar 2}, x_{4}\rangle \langle x_{\bar 4}, x_{2}\rangle\langle x_{1}, x_{3}\rangle
+\alpha\q^2 \langle x_{\bar 3}, x_{\bar 4}\rangle \langle x_{\bar 2}, x_{ 1}\rangle \langle x_{4}, x_{3}\rangle\langle x_{\bar 1}, x_{2}\rangle
\\
&\qquad\phantom{=}+ \s\langle x_{\bar 3}, x_{ 2}\rangle \langle x_{\bar 1}, x_{ \bar 4}\rangle \langle x_{ 4}, x_{ 1}\rangle\langle x_{\bar 2}, x_{3}\rangle
+\s\q \langle x_{\bar 3}, x_{ 1}\rangle \langle x_{\bar 2}, x_{\bar 4}\rangle \langle x_{ 4}, x_{2}\rangle\langle x_{\bar 1}, x_{3}\rangle
\\
&\qquad\phantom{=}+\s\langle x_{\bar 3}, x_{ 4}\rangle \langle x_{\bar 2}, x_{ \bar 1}\rangle \langle x_{\bar 4}, x_{3}\rangle\langle x_{ 1}, x_{2}\rangle
+\s^2\langle x_{\bar 3}, x_{ 2}\rangle \langle x_{\bar 1}, x_{ 4}\rangle \langle x_{ \bar 4}, x_{ 1}\rangle\langle x_{\bar 2}, x_{3}\rangle
\\
&\qquad\phantom{=}+\s^2\q \langle x_{\bar 3}, x_{ 1}\rangle \langle x_{\bar 2}, x_{4}\rangle \langle x_{\bar 4}, x_{2}\rangle\langle x_{\bar 1}, x_{3}\rangle
+\s^2q^2\langle x_{\bar 3}, x_{ 4}\rangle \langle x_{\bar 2}, x_{ 1}\rangle \langle x_{\bar 4}, x_{3}\rangle\langle x_{ \bar 1}, x_{2}\rangle.
\end{align*}
Since $\dim(H_{\R}) \geq 2$, there are two orthogonal unit eigenvectors $e_1$, $e_2$, and we take $x_{\bar 1}=x_4=e_1$, $x_{1}=x_{\bar 4}=e_2$, $x_{\bar{2}}=x_3=e_1$ and $x_{2}=x_{\bar{3}}=e_2$. Hence
\begin{align*}
&\state(\G(x_{\overline{ 4}}\otimes x_{ 4})\G(x_{\overline{ 3}}\otimes x_{ 3})\G(x_{\overline{ 3}}\otimes x_{ 2}) \G(x_{\bar 1}\otimes x_{ 1}))
\\
&\qquad{}
- \state(\G(x_{\overline{ 3}}\otimes x_{ 3})\G(x_{\overline{ 3}}\otimes x_{ 2}) \G(x_{\bar 1}\otimes x_{ 1})\G(x_{\overline{ 4}}\otimes x_{ 4}))=\s^2\q^2-\s^2.
\end{align*}
Therefore, the vacuum state is not a trace when $\s\neq 0$. When $\s= 0$,
the traciality follows from the similar arguments of \cite[Theorem 4.4]{BozejkoSpeicher1994}.
\end{proof}

For $\pi \in \NC_2^A(2 m)$ we say that B-pair is inner if it is covered by another B-pair. A B-pair of a non-crossing partition of type A is outer if it is not inner.
In accordance with this definition we can formulate the following corollary.

\begin{Corollary}\label{cor13} Assume that $x_{\bar i},x_i \in H_\R$ for $i=1,\dots,2m$.
\begin{enumerate}\itemsep=0pt
\item[$1.$] For $\s=0$, we recover the $\q$-deformed formula for moments {\rm \cite[Proposition~2]{BozejkoSpeicher1991}:}
\begin{align*}
&\state(G_{0,\q}(x_{\overline{2m}}\otimes x_{2m})\cdots G_{0,\q}(x_{\bar 1}\otimes x_1))
\\
& \qquad =\sum_{\substack{ \pi \in \P^{B}_2(2 m)\\ \NB(\pi)=0 }} \q^{\Cr(\pi)}\prod_{(i,j) \in\pi }\langle x_i, x_j\rangle
=\sum_{\pi \in \PA_2(2 m)} \q^{\Cr(\pi)}\prod_{(i,j) \in\pi }\langle x_i, x_j\rangle.
\end{align*}
\item[$2.$] For $\q=0$, we obtain the formula for moments of the symmetric Kesten's law
\begin{align*}
\state( {G}_{\s,0}(x_{\overline{2m}}\otimes x_{2m})\cdots {G}_{\s,0}(x_{\bar 1}\otimes x_1) )
=\sum_{\pi \in \NC_2^B(2 m)} \s^{\NB(\pi)}\prod_{(i,j) \in\pi }\langle x_i, x_j\rangle.
\end{align*}
\item[$3.$] For $\q=0$, we obtain the formula
\begin{align*}
\state({G}_{\s,0}(x_{{2m}}\otimes x_{2m})\cdots {G}_{\s,0}(x_{ 1}\otimes x_1))
= \sum_{\pi \in \NC_2^A(2 m)} (1+\alpha)^{\#\Out(\pi)}\prod_{(i,j) \in\pi }\langle x_i, x_j\rangle^2,
\end{align*}
where $\#\Out(\pi)$ is the number of outer B-pairs in a non-crossing partition of type A.
\end{enumerate}
\end{Corollary}

\begin{proof}
2. When $q=0$, then the nonzero terms in part~2 of Theorem \ref{lem101} occur only when $\pi$ is non-crossing of type B.

3. In order to proof the third point, we introduce the following map
\begin{align*}%\label{eq:odwzorowanie}
\begin{aligned}
\D\colon\ \NC_2^B(2 m)\to\NC_2^A(2 m),\qquad
\tilde \pi \mapsto \pi,
\end{aligned}
\end{align*}
where $\pi$ is defined through the
B-pairs of $\tilde \pi$; more precisely, the B-pairs $C=((\bar b,\bar a ), (a,b)) \in \Pair_B(\tilde \pi)$ are mapped according to the action of
the following relation
\begin{align*}
\begin{aligned}
C&\mapsto\begin{cases}
((\bar b,\bar a ), (a,b)), & \text{if $C$ is the positive pair of } \Pair_B(\tilde \pi),
\\[1mm]
((\bar b, a ), (\bar a,b)), & \text{if $C$ is the negative pair of } \Pair_B(\tilde \pi).
\end{cases}
\end{aligned}
\end{align*}

\begin{figure}[t]
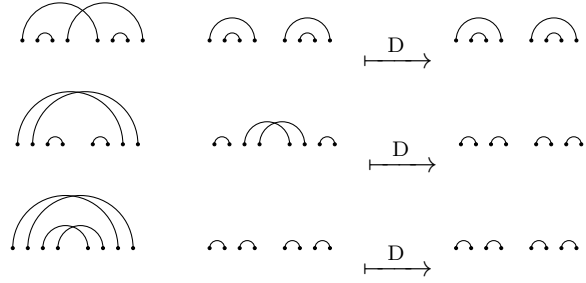

\centering
\MatchingMeandersexample{4}{-4/1, -3/-2,2/3,-1/4 }{}\hspace{0.7em} \MatchingMeandersexample{4}{-4/-1, -3/-2,2/3,1/4 }{}
$\xmapsto[\text{ }\text{ }\text{ }\text{ }\text{ }\text{ }\text{ }]{\D}$
\MatchingMeandersexample{4}{-4/-1, -3/-2,2/3,1/4 }{} \\
\MatchingMeandersexample{4}{-4/3, -2/-1,1/2,-3/4 }{} \hspace{0.7em} \MatchingMeandersexample{4}{-4/-3, -2/1,-1/2,3/4 }{} $\xmapsto[\text{ }\text{ }\text{ }\text{ }\text{ }\text{ }\text{ }]{\D}$
\MatchingMeandersexample{4}{-4/-3, -2/-1,1/2,3/4 }{}
\\
\MatchingMeandersexample{4}{-4/3, -2/1,-1/2,-3/4 }{} \hspace{0.7em} \MatchingMeandersexample{4}{-4/-3, -2/-1,1/2,3/4 }{}
$\xmapsto[\text{ }\text{ }\text{ }\text{ }\text{ }\text{ }\text{ }]{\D}$
\MatchingMeandersexample{4}{-4/-3, -2/-1,1/2,3/4 }{}
\caption{The visualization of the action $D$ on $\NC_2^B(4)$.}
\label{fig:exmapleouter}
\end{figure}

\noindent
We further observe that every negative B-pair in $\Pair_B(\tilde \pi)$ is mapped to an outer positive B-pair in $\Pair_B(\pi)$~-- see Figure~\ref{fig:exmapleouter}.
For any choice of outer blocks $V_1,\dots,V_i\in \pi$, there is a unique $\tilde\pi \in D^{-1}(\pi)$ which whose negative blocks are exactly the preimages of $V_1,\dots,V_i$.
Thus,
we get
\[
\#\D^{-1}(\pi)=\sum_{i=0}^{\#\Out(\pi)}{\#\Out(\pi) \choose i}\qquad\text{for}\quad\pi\in \NC_2^A(2 m).
\]
Using the above fact we obtain
\begin{align*}
\sum_{\pi \in \NC_2^A(2 m)} (1+\alpha)^{\#\Out(\pi)}\prod_{(i,j) \in\pi }\langle x_i, x_j\rangle^2 &=\sum_{\pi \in \NC_2^A(2 m)} \sum_{i=0}^{\#\Out(\pi)}\alpha^i{\#\Out(\pi) \choose i}\prod_{(i,j) \in\pi }\langle x_i, x_j\rangle^2
\\
&=\sum_{\pi \in \NC_2^A(2 m)}
\sum_{\tilde \pi \in\D^{-1}(\pi)}\alpha^{\NB(\tilde\pi)} \prod_{(i,j) \in\pi }\langle x_i, x_j\rangle^2
\intertext{and finally by using $\NC_2^B(2 m)=\bigsqcup_{\pi \in \NC_2^A(2 m)} \D^{-1}(\pi) $, we get}
&=\sum_{\tilde \pi \in \NC_2^B(2 m)} \s^{\NB(\pi)}\prod_{(i,j) \in\pi }\langle x_i,x_j\rangle.
\end{align*}
Let us observe that during this procedure we don't change the inner product because now we assume that $x_{\bar i }= x_i$.
\end{proof}

\begin{Remark}
The formula in Corollary~\ref{cor13}\,(3) was employed in many papers related to conditional free probability, e.g.,
\cite{BLS1996,BozejkoWysoczanski2001,Wojakowski2007}.
\end{Remark}

\subsection{The positive and negative inversions}
\label{positivenegativeinversions}
Finally, in this subsection, we explain the combinatorial interpretation of calculating the positive and negative inversions, which is well compatible with the procedure of counting crossings and negative blocks as described in
Remark~\ref{jakliczyc}.
Let us observe that we can rewrite a symmetrization operator by using the language of the positive and negative inversions; see \cite[pp.~219--220]{BozejkoSzwarc2003}. First, we need some definition of positive and negative roots of two types:
\begin{gather*}
\begin{split}
& R_1^+:=\{1,\dots, n \}\qquad \text{and}\qquad
R_1^-=-R_1^+,
\\
& R_2^+:=\{(i,j)\mid 1\leq i < j\leq n \}\cup \{(i,-j)\mid 1\leq i < j\leq n \}\qquad
\text{and}\qquad
R_2^-:=-R_2^+.
\end{split}
\end{gather*}

In this style, we may express the lengths $l_1$ and $l_2$ as
\begin{itemize}\itemsep=0pt
\item $l_1(\sigma)=\operatorname{ninv}(\sigma):=\operatorname{card}\{i\mid 1\leq i \leq n\text{ and } \sigma(i )<0 \}$ is the number of negative inversions of $\sigma\in B(n)$;
\item $l_2(\sigma)=\operatorname{pinv}(\sigma):=\operatorname{card}\{(i,j)\in R_2^+\mid (\sigma(i),\sigma(j))\in R_2^- \}$ is the number of positive inversions of $\sigma\in B(n)$;
\end{itemize}
With the above notation, we have
$P_{\s,\q}^{(n)}= \sum_{\sigma \in B(n)}\s^{\operatorname{ninv}(\sigma)} \q^{\operatorname{pinv}(\sigma)} \sigma$.
\begin{Remark}

The lengths functions $l_1$ and $l_2$ are very related to root system
of the Coxeter group of type B.
We can use results of \cite[Proposition~1]{BozejkoSzwarc2003}
or \cite[Chapter~4.3]{Bourbaki}.
In our case of the group $B(n)$, the this is related to root system of
type B
\[
\Pi=\Pi_1\cup \Pi_2,
\]
where $\Pi_1=\{e_1,\dots,e_n\}$ and $\Pi_2=\{e_i\pm e_j\mid i<j\}$.
The related length functions (see \cite[Proposition 1]{BozejkoSzwarc2003}) on our group $B(n)$ are the following:
\[
l_i (\sigma) = \#\big\{ \Pi_i \cap \sigma^{-1} (- \Pi _i )\big\}, \qquad i=1,2,
\]
which
in our notations reads
$R_2^{+}=\Pi_{2}$ and $R_1^{+} = \Pi_{1}$.
More information on the subject can be found in the books of \cite{BjornerBrenti,Humphreys}.
\end{Remark}

This new definition of length has the following interpretation.
We draw arrows which show the action of permutation; see Figure~\ref{permutacjeexmpale}.
Then we follow the rules:
\begin{itemize}\itemsep=0pt
\item if an arrow $a \rightarrow b$ crosses $-a \rightarrow -b$, then we count them as a negative inversion. This crossing appears in the center of picture and corresponds to a crossing which appears when we draw the negative pair $(\bar b,\bar a )\otimes (a,b)$, i.e., the cross of $(\bar b,\bar a )$ and $ (a,b)$.
\item if an arrow $a \rightarrow b$ crosses $c \rightarrow d$, ($c\neq -a$) and repeatedly, $-a \rightarrow -b$ crosses $-c \rightarrow -d$, then we count one of them as a positive inversion. In other words, we count positive inversions as crossings of arrows, which lie to the left or to the right of the center of the picture.
\end{itemize}
We would like to emphasize that the same rule applies when we count corresponding statistics for the set of pair partitions of type B.

\begin{figure}[htp]
\centering
\begin{tikzpicture}[node distance={15mm}, thick, main/.style = {}]
%\node[main] (1) {$-3$};
\node[main] (2) {$-2$};
\node[main] (3) [ right of=2] {$-1$};
\node[main] (4) [ right of=3] {$1$};
\node[main] (5) [ right of=4] {$2$};
\node[main] (6) [ right of=5] { };

%\node[main] (a) [below of=1] {$-3$};
\node[main] (b) [below of=2] {$-2$};
\node[main] (c) [ below of=3] {$-1$};
\node[main] (d) [ below of=4] {$1$};
\node[main] (e) [ below of=5] {$2$};
\node[main] (k) [ below of=e] {};
\node at (7.5, 0) (h) {$\sigma=
\bigl(\begin{smallmatrix}
-2 & -1 & 1 & 2 \\
2 & -1 & 1 & -2
\end{smallmatrix}\bigr)=\pi_1\pi_0\pi_1
$};
\node at (7.5, -1.5) {${\operatorname{inv}(\sigma)=2,\hspace{0.5em} \operatorname{ninv}((\sigma)=1}$ };
%\node at (7,-1.5) {\MatchingMeandersab{2}{-4/3, -2/-1,1/2,-3/4 }{}};

%\node[main] (i) [ below of=h] {$\operatorname{ninv}((\sigma)=1$};
%\node[main] (f) [ below of=6] {$3$};
\draw[->] (2) -- (e);
\draw[->] (5) -- (b);
\draw[->] (3) -- (c);
\draw[->] (4) -- (d);
\end{tikzpicture}

\vspace{-5mm}

\begin{tikzpicture}[node distance={15mm}, thick, main/.style = {}]
\node[main] (2) {$-2$};
\node[main] (3) [ right of=2] {$-1$};
\node[main] (4) [ right of=3] {$1$};
\node[main] (5) [ right of=4] {$2$};
\node[main] (6) [ right of=5] { };

%\node[main] (a) [below of=1] {$-3$};
\node[main] (b) [below of=2] {$-2$};
\node[main] (c) [ below of=3] {$-1$};
\node[main] (d) [ below of=4] {$1$};
\node[main] (e) [ below of=5] {$2$};
\node[main] (k) [ below of=e] {};
\node at (7.5, -1.5) (h) {${\operatorname{inv}(\sigma)=1,\hspace{0.5em} \operatorname{ninv}((\sigma)=2}$ };
\node at (7.6,-0) (h) {$\sigma=
\bigl(\begin{smallmatrix}
-2 & -1 & 1 & 2 \\
1 & 2 & -2 & -1
\end{smallmatrix}\bigr)=\pi_0\pi_1\pi_0
$ };
\draw[->] (2) -- (d);
\draw[->] (5) -- (c);
\draw[->] (3) -- (e);
\draw[->] (4) -- (b);
\end{tikzpicture}

\vspace{-5mm}

\begin{tikzpicture}[node distance={15mm}, thick, main/.style = {}]
\node[main] (2) {$-2$};
\node[main] (3) [ right of=2] {$-1$};
\node[main] (4) [ right of=3] {$1$};
\node[main] (5) [ right of=4] {$2$};
\node[main] (6) [ right of=5] { };

\node at (7.5, -1.5) (h) {${\operatorname{inv}(\sigma)=1,\hspace{0.5em} \operatorname{ninv}((\sigma)=0}$ };
\node at (7,-0) (h) {$\sigma=
\bigl(\begin{smallmatrix}
-2 & -1 & 1 & 2 \\
-1 & -2 & 2 & 1
\end{smallmatrix}\bigr)=\pi_1
$ };

%\node[main] (a) [below of=1] {$-3$};
\node[main] (b) [below of=2] {$-2$};
\node[main] (c) [ below of=3] {$-1$};
\node[main] (d) [ below of=4] {$1$};
\node[main] (e) [ below of=5] {$2$};
%\node[main] (k) [ below of=e] {};
\draw[->] (2) -- (c);
\draw[->] (3) -- (b);
\draw[->] (4) -- (e);
\draw[->] (5) -- (d);
\end{tikzpicture}
\caption{The example of a permutation $B(2)$ with inversions.}
\label{permutacjeexmpale}
\end{figure}
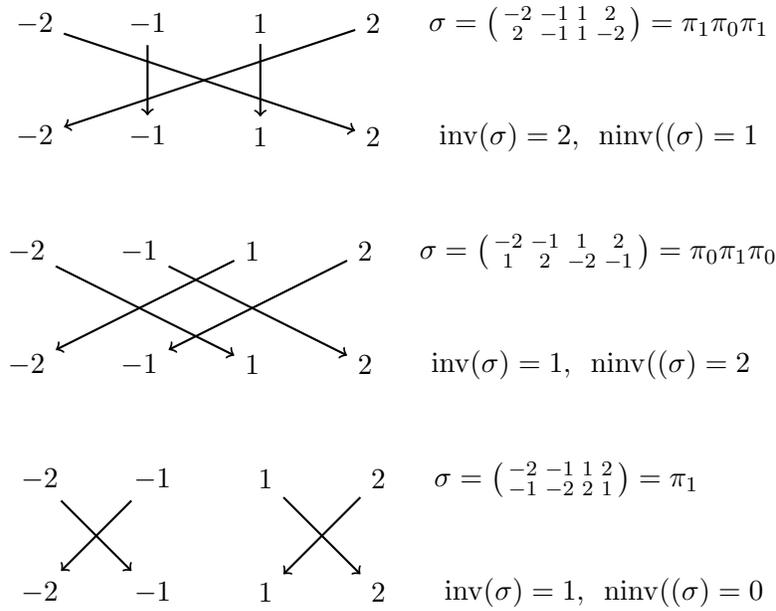

\begin{Remark}%\label{crossingofarrow}
We say that $\sigma\in B(n)$ has a crossing of an arrow if it has a crossing in a sense described and illustrated in~Figure~\ref{permutacjeexmpale}, and we denote the cardinality of them by $\Cr(\sigma). $
\end{Remark}

\subsection{The Coxeter arcsine distribution}
The Catalan number
$\operatorname{Cat}(W)$ is defined for any finite Coxeter group $W$.
Its
explicit formula is
\begin{align*}
\operatorname{Cat}(W)=\frac{1}{|W|}\prod_{i=1}^n(h+d_i),
\end{align*}
where $h$ is the Coxeter number and $d_1,\dots,d_n$ are the degrees of $W$, arising
from its ring of polynomial invariants (see \cite[Section 2.7]{Armstrong}). The number $\operatorname{Cat}(W)$ has
been discovered independently in several different areas and wherever it appears, it
is accompanied by a wealth of new combinatorics.
Armstrong \cite[p.~39, Table 2.8]{Armstrong} displays the complete list of Coxeter--Catalan numbers for finite irreducible Coxeter groups;
in type B it is a central binomial coefficient
$
{2n}\choose{n}$.
This sequence is cataloged as A000984 in Sloane’s database \cite{Sloane} and the numbers are
called the \emph{Catalan numbers of type B}.

\begin{Definition}
A type-$B$ set partition in a sense of Reiner \cite[Section 2]{Reiner1997} is a set partition~$\P_R^{B}(n) $ %\P^{\mathcal{B}}(n)
of the set $[\pm n]$, which is ordered according to the principle $-1<\dots<-n<1<\dots<n$ satisfying the following two conditions:
\begin{itemize}\itemsep=0pt
\item if $B$ is a block in $\P_R^{B}(n)$, then $-B$ is also a block in $\P_R^{B}(n)$;
\item there exists at most one block $B$ in $\P_R^{B}(n)$ for which $B =-B$ called the \emph{zero block.}
\end{itemize}
If it is non-crossing then we call it a non-crossing partition. We denote by $\NC^{B}_R(n)$ the set of all non-crossing partitions in a sense of Reiner.
\end{Definition}

Our definition of partition in $\P_{1,2}^{B}(n)$ is quite similar to a partitions $\P_R^{B}(n)$, but in our situation the zero block does not exist.

Reiner \cite{Reiner1997} showed that
the cardinality of the set of non-crossing partitions of type $W$ on Coxeter group $W$ is $\operatorname{Cat}(W)$ and in the type-B situation we have:
$\#\NC^{{B}}_R(n)={{2n}\choose{n}}$.
It is worth to mention that the authors also refer to these numbers as the type-B Catalan numbers, because they count the number of lattice paths of type B from $(0,0)$ to $(n,n)$ using steps $(1,0)$ and $(0,1)$; see \cite{Stump2008, Stump2010}.

Finally, we count the number of non-crossing pair partitions of type B, which appear in this article. If we put $\q=0$ and $\s=1 $ in \eqref{recursion}, then we obtain
\begin{align*}
t Q_n^{(1, 0)}(t) &= Q_{n+1}^{(1, 0)}(t) +2Q_{n-1}^{(1, 0)}(t) , \qquad n=1,
\\ t Q_n^{(1, 0)}(t) &= Q_{n+1}^{(1, 0)}(t) +Q_{n-1}^{(1, 0)}(t), \qquad n=2,3,\dots.
\end{align*}
These are the Chebyshev polynomials of the first kind, defined through the identity \[Q_{n}^{(1, 0)}(2\cos(\theta))=\cos(n\theta).\]
These polynomials are well known and the orthogonalizing probability measure is a symmetric arcsine distribution $\frac{1}{\pi \sqrt{4-x^2}}$, supported on $(-2, 2)$; see \cite[p.~39]{HoraObata}. The moments of this distribution is a central binomial coefficient; from this and from Corollary~\ref{cor13}, we conclude that
\[
\#\NC_2^B(2 n)={{2n}\choose{n}}.
\]
It seems to us that this is the first interpretation of the Catalan numbers of type B
as pair partitions coming form our double Fock space of type B study in that paper. A few years ago Professor P.~Biane let us know that this case is interesting
because the arcsine distribution should be analogue of the normal distribution in a free probability of type~B.

\begin{Remark}
Biane, Goodman and Nica \cite{BianeGoodmanNica2003} showed that it is possible to build a free probability theory of type B, by replacing the occurrences of the symmetric groups and the non-crossing partitions of type A by their type B analogues and the non-crossing partitions of Reiner \cite{Reiner1997}, i.e., $\NC^{B}_R(n)$. In their work,
a central role is played by the boxed convolution, which is a~combinatorial operation having a natural type B analogue and describing the multiplication of
two freely independent non-commutative random variables.
Our construction of double Fock space of type B
uses the geometry and length function of Coxeter groups based on simple roots, but
the main idea of the paper \cite{BianeGoodmanNica2003} is to use all roots and another length function. Recently,
connections between construction \cite{BianeGoodmanNica2003} and $c$-freeness and infinitesimal free probability were put into evidence in
\cite{BelinschiShlyakhtenko2012,CebronDahlqvistGabriel2022,CebronGilliers2022,Fevrier2012,EewierSzpojankowskiNicaMastnak2019,
Popa2010}.
\end{Remark}

\section{Fock space in edge cases}
In the last section, we explain a situation when
the kernel of the symmetrization $P_{\s,\q}^{(n)}$ is nontrivial.
In this case ${P_{\s,\q}^{(n)}}$ projects to the space of special symmetric and antisymmetric tensors. In the group algebra $\C G$ the natural involution is defined by $x^\ast=x^{-1}$ for $x\in G$ and extended by anti-linearity for all groups algebra.

\begin{Lemma}\label{lematprojektr}
If $H\neq\{e\}$ is a subgroup of a finite group $G$ and function $\varphi$ is a character of $H$, then
$P_H=\frac{1}{|H|}\sum_{x\in H}\varphi(x)x$ is an orthogonal projection.
\end{Lemma}
\begin{proof}
Let us first observe that $P_H=P_H^\ast$ because $\overline{\varphi(x)}=\varphi\big(x^{-1}\big)$ for $x\in H$. We also have
\begin{align*}
|H|^2P_H^2&=\sum_{x\in H}\sum_{y\in H}\varphi(x)\varphi(y)xy=\sum_{x\in H}\sum_{y\in H}\varphi(xy)xy
\\
&=\sum_{x\in H}\sum_{t\in H}\varphi(t)t=| H|\sum_{t\in H}\varphi(t)t=|H|^2P_H.
\end{align*}
From this we see that $P_H$ is an orthogonal projection.
\end{proof}

\begin{Remark}
Assume that $(\s,\q)\in\{(\pm 1,\pm 1)\}$.
Then we have that $P_{\s,\q}^{(n)}:= \sum_{\sigma \in B(n)}\varphi_{\s,\q}(\sigma)\sigma$, where $\varphi_{\s,\q}(\sigma):= \s^{l_1(\sigma)} \q^{l_2(\sigma)} $
is a character of the group $B(n)$.
Indeed, we can easily see that in each of the four cases $(\s,\q)\in\{(\pm 1,\pm 1)\}$, we have
\[
\varphi_{\s,\q}(\sigma\gamma)=\s^{l_1(\sigma\gamma)} \q^{l_2(\sigma\gamma)}=\s^{l_1(\sigma)} \q^{l_2(\sigma)}\s^{l_1(\gamma)}\q^{l_2(\gamma)}=\varphi_{\s,\q}(\sigma)\varphi_{\s,\q}(\gamma).
\]

If these cases $(\s,\q)\in\{(0,\pm 1),(\pm 1,0)\}$, we obtain also two characters of subgroups $H_0=gp\{\pi_0\}$ and two characters of the permutation group $S(n) $ generated by $\{\pi_1,\dots,\pi_{n-1}\}$.
\end{Remark}

Now we will specify the range of parameters for which the considered operator is invertible.

\begin{Lemma} %\label{lemathom}
If $|\q|< 1$, $|\s|< 1$ and $(\s,\q)\notin\{(\pm 1,\pm 1),(0,\pm 1),(\pm1,0)\}$, then considered sym\-metriza\-tion ${P_{\s,\q}^{(n)}}$ is invertible and $\ker {P_{\s,\q}^{(n)}}=\{0\}$, as operator on $l^2(B(n))$ and on $\H=H^{\otimes 2n}$, for arbitrary
Hilbert space $H$.
\end{Lemma}
\begin{proof}
The proof of this lemma follows from \cite[Theorem 2.4]{BozejkoSpeicher1994}.
\end{proof}

Our operator $P^{(n)}_{\s,\q}$ has a nontrivial kernel in the following situation.

\begin{Theorem}
If $\s\in[-1,1]$ and $q=\pm1$ or $\q\in[-1,1]$ and $\s=\pm1$, then the operator $P^{(n)}_{\s,\q}$ has a nontrivial kernel.
\end{Theorem}

\begin{proof}
First we consider the case $\s\in[-1,1]$ and $q=\pm1$. Then by well-known result from Coxeter groups (see Remark~\ref{uwagibozejko} below), we have
\[
P_{\s,\pm1}^{(n)}=T_1P_{0,\pm1}^{(n)},
\]
and the operator $\frac{1}{n!}P_{0,\pm1}^{(n)}$ is the orthogonal projection (see Lemma~\ref{lematprojektr}). Therefore the space
\[
V_{0,\pm1}^{(n)}=\bigg(\id-\frac{1}{n!}P_{0,\pm1}^{(n)}\bigg)(\H)
\]
is the kernel of $P_{\s,\pm1}^{(n)}$ since
\[
P_{\s,\pm1}^{(n)}\big(V_{0,\pm1}^{(n)}\big) =\frac{1}{n!}T_1P_{0,\pm1}^{(n)} \bigg(\id-\frac{1}{n!}P_{0,\pm1}^{(n)}\bigg)(\H)=\{0\}.
\]
We proceed in a similar manner when $\q\in[-1,1]$ and $\s=\pm1$. Then we get form Remark~\ref{uwagibozejko}
\[
P_{\pm1,\q}^{(n)}=T_2 \frac{1}{2}P_{\pm1,0}^{(n)}.
\]
Since $\frac{1}{2}P_{\pm1,0}^{(n)}=\frac{1}{2}(e+\pm \pi_0)$ is the orthogonal projection on $l^2(H_0)$, where $H_0=gp\{\pi_0\}$. Hence we get that on the subspace
\[
W_{\pm1,0}^{(n)}=\bigg(\id-\frac{1}{2}P_{\pm1,0}^{(n)}\bigg)(\H),
\]
we have $P_{\pm1,0}^{(n)}\big(W_{\pm1,0}^{(n)}\big)=\{0\}$.
\end{proof}

\begin{Remark}\label{uwagibozejko}\quad
\begin{enumerate}
\item[(1)] We use very important result form the Coxeter groups $(W,S)$. If $J\subset S$ and $W_J=gp(J)$, then there exists decomposition
\[
W=W^JW_J
\]
(see \cite{Humphreys}). Moreover, if we define for the subset $A\subset W$
\[
P_{\underline q}^{(n)}(A)=\sum_{w\in A}{\underline q}^{l(w)}w,
\]
then $P_{\underline q}^{(n)}(W)=P_{\underline q}^{(n)}(W_J)P_{\underline q}^{(n)}\big(W^J\big)$.
In our case, $\underline q=(\s,\q)$ and ${\underline q}^{l(w)}={\alpha}^{l_1(w)}{ \q}^{l_2(w)}$.

\item[(2)] We also observe that if $(\s,\q)\in\{(\pm 1,\pm 1)\}$, then $P_{\s,\q}^{(n)}$ is homomorphism of $B(n)$, i.e., $P_{\s,\q}^{(n)}(\sigma_1\sigma_2)=P_{\s,\q}^{(n)}(\sigma_1)P_{\s,\q}^{(n)}(\sigma_2)$, Therefore, the natural extension of $P_{\s,\q}^{(n)}$ to tensor product of Hilbert space $\tilde P_{\s,\q}^{(n)}\colon \H \to \H $ as was presented in beginning of Section 3. Hence we have identification of linear spaces
\[
\H/ \ker \tilde P_{\s,\q}^{(n)}\cong \tilde P_{\s,\q}^{(n)}(\H ).
\]

\item[(3)]
Finally, we describe four special tensor products.
\emph{The type B-tensor product} is of the form
\begin{align*}
x_{\overline n}\tilde \otimes\cdots \tilde \otimes x_{n}:&=\frac{1}{n!2^n}\tilde P_{\pm 1,\pm 1}^{(n)}x_{\overline n} \otimes\cdots \otimes x_{n}.
\end{align*}
\end{enumerate}
In particular cases, we distinguish the following tensors:
\begin{enumerate}[(I)]\itemsep=0pt
\item If $\s=\q= 1$, then symmetrization covers all \emph{B-symmetric tensor products}, i.e.,
\begin{align*}
x_{\overline n}\tilde \otimes\cdots \tilde \otimes x_{n}:&=\frac{1}{n!2^n}\sum_{\sigma \in B(n)} x_{\sigma(\overline n)} \otimes\cdots \otimes x_{\sigma(n)},
\intertext{i.e., we have }
\sigma(x_{\overline n}\tilde \otimes\cdots \tilde \otimes x_{n})&=x_{\overline n}\tilde \otimes\cdots \tilde \otimes x_{n}.
\end{align*}
\item If $\s= -1$ and $\q=-1$, then we have the antisymmetric tensor product of type B, i.e.,
\begin{align*}
x_{\overline n}\tilde \otimes\cdots \tilde \otimes x_{n}=\frac{1}{n!2^n}\sum_{\sigma \in B(n)}(-1)^{\text{number of inversions in }\sigma} x_{\sigma(\overline n)} \otimes\cdots \otimes x_{\sigma(n)}.
\end{align*}
\item If $\s= 1$ and $\q=-1$, then we have the \emph{fermionic tensor product of type B}, i.e.,
\begin{align*}
x_{\overline n}\tilde \otimes\cdots \tilde \otimes x_{n}=\frac{1}{n!2^n}\sum_{\sigma \in B(n)}(-1)^{\text{number of positive inversions in }\sigma}
x_{\sigma(\overline n)} \otimes\cdots \otimes x_{\sigma(n)}.
\end{align*}
\item If $\s= -1$ and $\q=1$, then we have the \emph{bosonic tensor product of type B}, i.e.,
\begin{align*}
x_{\overline n}\tilde \otimes\cdots \tilde \otimes x_{n}=\frac{1}{n!2^n}\sum_{\sigma \in B(n)}(-1)^{\text{number of negative inversions in }\sigma}
x_{\sigma(\overline n)} \otimes \cdots \otimes x_{\sigma(n)}.
\end{align*}
Next we distinguish yet four situations $(\s,\q)\in\{(0,\pm 1),(\pm1,0)\}$, namely
\item
If $\q=0$ and $\alpha=\pm1$, then $P_{\alpha,0}^{(n)}=1+\alpha\pi_0$ and $\frac{1}{2} \tilde P_{\alpha,0}^{(n)}$ is a projection. From this we see that
\begin{enumerate}\itemsep=0pt
\item if $\q=0$ and $\s= -1$, then we don't have a positive inversion and at most one negative inversion $\pi_0$ and so we obtain
\begin{align*}
x_{\overline n}\tilde \otimes\cdots \tilde \otimes x_{n}={}& \frac{1}{2}x_{\overline n} \otimes \cdots \otimes x_{\bar 1} \otimes x_{ 1} \otimes \dots \otimes x_{n}
\\ &-\frac{1}{2}x_{\overline n} \otimes \cdots \otimes x_{ 1} \otimes x_{\bar 1}\otimes \dots \otimes x_{n} .
\end{align*}
This tensor product might be called the \emph{Boolean tensor product of type B};
\item if $\q=0$ and $\s= 1$, then we obtain the \emph{free tensor product of type B}
\begin{align*}
x_{\overline n}\tilde \otimes\cdots \tilde \otimes x_{n}={}& \frac{1}{2}x_{\overline n} \otimes \cdots \otimes x_{\bar 1} \otimes x_{ 1} \otimes \dots \otimes x_{n}
\\ &+\frac{1}{2}x_{\overline n} \otimes \cdots \otimes x_{ 1} \otimes x_{\bar 1}\otimes \dots \otimes x_{n}.
\end{align*}
\end{enumerate}

\item If $\q=\pm 1$ and $\s= 0$, then we have two cases namely a negative inversion does not appear and we obtain the classical fermionic relation for $q=-1$ and the bosonic relation for $q=1$, namely
\begin{align*}
x_{\overline n}\tilde \otimes\cdots \tilde \otimes x_{n}&=\frac{1}{n!}\sum_{\substack{ \sigma \in B(n) \\ \operatorname{ninv}(\sigma)=0 }} (\pm 1)^{\text{number of positive inversions in }\sigma}
x_{\sigma(\overline n)} \otimes \cdots \otimes x_{\sigma(n)}.
\end{align*}
\end{enumerate}
\end{Remark}
The above cases are interesting and will be considered in future work. The most significant situations appear, when $\q=0$ and $\s= 1$ as it can be a starting point for creating a new kind of type B free probability.

\subsection*{Acknowledgements}
	
M.~Bo\.zejko is supported by Narodowe Centrum Nauki, grant NCN 2016/21/B/ST1/00628.
W.~Ejsmont was supported by the Narodowe Centrum Nauki, grant No 2018/29/B/HS4/01420.

\pdfbookmark[1]{References}{ref}
\LastPageEnding

\end{document}